\documentclass[oneside]{amsart}

\usepackage{caption, fullpage}
\usepackage{subcaption}
\usepackage{pifont, xcolor, tikz}
\usepackage{algorithm,algorithmicx,algpseudocode}
\usepackage{amsmath}
\usepackage{ytableau}
\usepackage{array}
\usepackage{amssymb}
\usepackage{enumitem}

\newtheorem{theorem}{Theorem}
\newtheorem{proposition}[theorem]{Proposition}
\newtheorem{prop}[theorem]{Proposition}
\newtheorem{lemma}[theorem]{Lemma}
\newtheorem{conjecture}[theorem]{Conjecture}
\newtheorem{example}[theorem]{Example}

\newtheorem{conj}{Conjecture}
\newtheorem{corollary}[theorem]{Corollary}

\newcommand{\oz}{\ensuremath{\overline{z}}}

\newcommand{\bz}{\ensuremath{\mathbf{z}}}

\def\cT{\mathcal{T}}
\def\cS{\mathcal{S}}
\def\cD{\mathcal{D}}

\DeclareMathOperator{\futne}{{ne_\leftarrow}}
\DeclareMathOperator{\enhne}{{ne_\downarrow}}
\newcommand{\yt}[1]{
\ensuremath{\begin{array}{c}
\begin{ytableau}#1\end{ytableau}
\end{array}}
}

\usetikzlibrary{positioning}
\usetikzlibrary{decorations.pathreplacing}

\tikzset{partition/.style={fill,circle,inner sep=1pt},
         part/.style={baseline=0,scale=0.5,bend left=45},
         partlabel/.style={below}}
\usetikzlibrary{shapes,arrows}
\tikzstyle{pnt}=[draw,ellipse,fill,inner sep=1pt]

\tikzstyle{opnt}=[ ]
\tikzstyle{pntt}=[draw,ellipse,fill,inner sep=0.5pt]
\tikzstyle{point}=[draw,ellipse,fill,inner sep=2pt]

\makeatletter
\newcommand{\oset}[2]{%
  {\mathop{#2}\limits^{\vbox to \ex@{\kern-\tw@\ex@
   \hbox{\ensuremath{#1}}\vss}}}}
\makeatother

\author{Sophie Burrill
\and Julien Courtiel
\and Eric Fusy
\and Stephen Melczer
\and Marni Mishna
}
\address{
\noindent S. Burrill, Department of Mathematics, Simon Fraser University,
Burnaby BC Canada. \newline
J. Courtiel, Department of Mathematics, Simon Fraser University,
Burnaby BC Canada. \newline
E. Fusy, LIX, \'Ecole Polytechnique, Palaiseau, France. \newline
S. Melczer, Cheriton School of Computer Science, University of Waterloo, Waterloo
ON Canada. 
\&  U. Lyon, CNRS, ENS de Lyon, Inria, UCBL, Laboratoire LIP\newline
M. Mishna, Department of Mathematics, Simon Fraser University,
Burnaby BC Canada. }

\title{Tableau sequences, open diagrams, and Baxter families }


\begin{document}

\begin{abstract}
  Walks on Young's lattice of integer partitions encode many objects
  of algebraic and combinatorial interest. Chen \emph{et al.\/}
  established connections between such walks and arc diagrams. We show
  that walks that start at $\varnothing$, end at a row shape, and only
  visit partitions of bounded height are in bijection with a new type
  of arc diagram -- open diagrams. Remarkably two subclasses of open
  diagrams are equinumerous with well known objects: standard Young
  tableaux of bounded height, and Baxter permutations. We give an
  explicit combinatorial bijection in the former case.
\end{abstract}

\maketitle


\section{Introduction}
\label{sec:Introduction}
The lattice of partition diagrams, where domination is given by
inclusion of Ferrers diagrams, is known as Young's lattice. Walks on
this lattice are of significant importance since they encode many
objects of combinatorial and algebraic interest. 

At a first level, a walk on Young's lattice is a sequence of Ferrers diagrams such that at most a
box is added or deleted at each step. A class of such sequences is also
known as a \textit{tableau family}.  There are several combinatorial classes in
explicit bijection with tableau families ending in an empty shape, in
particular when there are restrictions on the height of the tableaux
which appear.

In this work we study tableau families that encode walks in Young's
lattice that start at the empty partition and end with a partition
composed of a single part:~$\lambda=(m)$, $m\geq 0$.  Additionally, they are
bounded, meaning that they only visit partitions that have at most $k$
parts, for some fixed~$k$. In particular, we generalize the results of
Chen~\emph{et al.}~\cite{Chetal07}, and Bousquet-M\'elou and
Xin~\cite{BoXi05} to prove that two classic combinatorial classes
--Young tableaux of bounded height and Baxter permutations -- are in
bijection with bounded height tableau families.

\subsection{Part 1. Oscillating tableaux and Young tableaux of bounded height}
The first tableau family that we study is the set of \emph{oscillating
  tableaux\/} with height bounded by~$k$. These appear in the study of
partitions avoiding certain nesting and crossing
patterns~\cite{Chetal07}. Our first main result is a bijection
connecting oscillating tableaux to the class of standard Young tableau
of bounded height. Young tableaux are more commonly associated with
oscillating tableau, and ours is a very different connection.  This
result demonstrates a new facet of the ubiquity of Young tableaux.

\begin{theorem}
\label{thm:SYT}
The set of oscillating tableaux of size~$n$ with height bounded
by~$k$, which start at the empty partition and end in a row shape
$\lambda=(m)$, is in bijection with the set of standard Young tableaux
of size $n$ with height bounded by $2k$, with $m$ odd columns.
\end{theorem}
The proof of Theorem~\ref{thm:SYT} is an explicit bijection between
the two classes. One consequence of the bijective map is the symmetric
joint distribution of two kinds of nesting patterns inside the class
of involutions.

Enumerative formulas for Young tableaux of bounded height have been
known for almost half a century~\cite{Gord71,GoHo68, BeKn72}, but new
enumerative formulas can be derived from Theorem~\ref{thm:SYT},
notably an expression which can be written as a diagonal of a
multivariate rational function. The new generating function
expressions are the subject of Section~\ref{sec:sytbh-egf}.

\subsection{Part 2. Hesitating tableaux and Baxter Permutations}
In the second part, we consider the family of \emph{hesitating
  tableaux}. These tableau sequences appear in studies of set
partitions avoiding so-called \emph{enhanced\/} nesting and crossing
patterns. The work of~\cite{Chetal07} again serves to describe
bijections between lattice paths and arc diagrams. Using a lattice
path interpretation, we make a generating function argument to connect
this combinatorial class to Baxter permutations. This connection was
recognized by Xin and Zhang~\cite{XiZh09}. Here we offer an explicit
proof, using formulas of Bousquet-M\'elou and Xin~\cite{BoXi05}, of
the following result.
\begin{theorem}
\label{thm:Baxter}
The number of hesitating tableaux of length $2n$ of height strictly
less than three is equal to the number~$B_{n+1}$ of Baxter
permutations of length~$n+1$, where
\[
B_n=\sum_{k=1}^n\frac{\binom{n+1}{k-1}\binom{n+1}{k}\binom{n+1}{k+1}}{\binom{n+1}{1}\binom{n+1}{2}}.
\] 
\end{theorem}
Baxter numbers have been described as the ``big brother'' of the well
known Catalan numbers: they are the counting series for many
combinatorial classes, and these classes often contain natural
subclasses which are counted by Catalan numbers. For example, doubly
alternating Baxter permutations have a Catalan number counting
sequence~\cite{GuLi00}. One consequence of Theorem~\ref{thm:Baxter} is
a new two variable generating tree construction for Baxter numbers.

Unlike the results in Part 1, our proof of Theorem~\ref{thm:Baxter} is
not a combinatorial bijection. One impediment to a bijective proof is
a lack of a certain symmetry in the class of hesitating tableaux that
is present in most known Baxter classes. A bijection would certainly
be of interest, and in fact we conjecture a refinement of
Theorem~\ref{thm:Baxter}, in Conjecture~\ref{thm:BaxterConjecture},
which could guide a combinatorial bijection. 

We begin with definitions in Section~\ref{sec:combclasses}, and
some known bijections. Then we focus on the standard Young Tableaux of
bounded height in Section~\ref{sec:YTBH}, followed by our study of Baxter
objects in Section~\ref{sec:Baxter}.

\section{The combinatorial classes}
\label{sec:combclasses}
We begin with precise definitions for the combinatorial classes that
are used in our results.
\subsection{Tableaux families}
As mentioned above, a common encoding of walks on Young's lattice is given
by sequences of Ferrers diagrams. We consider three variants. 
Each sequence starts from the empty shape, and has a specified ending shape;
the difference between them is the limitations they
impose on when one can add or remove a box. The \emph{length\/}
of a sequence is the number of elements, minus one. (It is the number
of steps in the corresponding walk.)

\begin{description}
\item[A vacillating tableau] is an \emph{even} length sequence of
  Ferrers diagrams $(\lambda^{(0)}, \dots, \lambda^{(2n)})$ where
  consecutive elements in the sequence are either the same or differ
  by one square, under the restriction that $\lambda^{(2i)}\geq
  \lambda^{(2i+1)}$ and $\lambda^{(2i-1)}\leq \lambda^{(2i)}$.
\item[A hesitating tableau] is an even length sequence of Ferrers
  diagrams $(\lambda^{(0)}, \dots, \lambda^{(2n)})$ where consecutive
  differences of elements in the sequence fall under one of the
  following categories\footnote{Recall $\lambda\leq\mu$ means that $\lambda_i\leq\mu_i$ for all $i$} :
\begin{itemize}
\item $\lambda^{(2i)} = \lambda^{(2i+1)}$ and   $\lambda^{(2i+1)} <
  \lambda^{(2i+2)}$ (do nothing; add a box)
\item  $\lambda^{(2i)} > \lambda^{(2i+1)}$and   $\lambda^{(2i+1)} =
  \lambda^{(2i+2)}$ (remove a box; do nothing)
\item $\lambda^{(2i)} < \lambda^{(2i+1)}$ and   $\lambda^{(2i+1)} >
  \lambda^{(2i+2)}$ (add a box; remove a box).
\end{itemize}  
\item[An oscillating tableau] is simply a sequence of Ferrers diagrams
  such that at every stage a box is either added or deleted. Remark that
  the length of the sequence is not necessarily even.
\end{description}
In each case, if no diagram in the sequence is of height $k+1$, we say
that the tableau has its \emph{height bounded
  by~$k$}. Figure~\ref{fig:tableaux} shows examples of the different
tableaux.

\ytableausetup{smalltableaux}
\begin{figure}
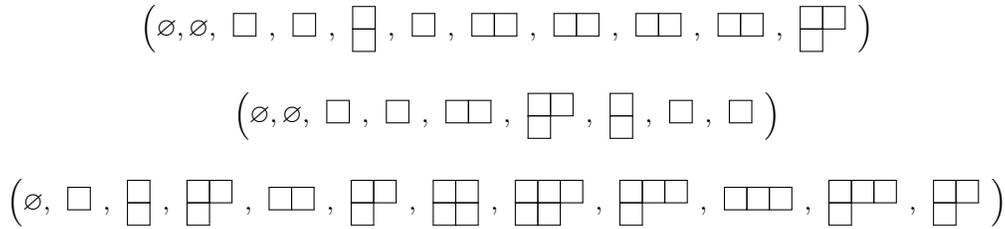


\center
$$\left(\varnothing,  \varnothing,
\yt{\ },  \yt{\ }, \yt{\  \\ \ }, \yt{\   }, \yt{\ & \ },\yt{\ & \ },\yt{\ & \ }, \yt{\  & \  } , \yt{\  & \ \\ \ } \right)$$

$$\left(\varnothing, \varnothing, 
\yt{\ },  \yt{\ }, \yt{\  & \ }, \yt{\ & \ \\ \ },\yt{\ \\ \ }, \yt{\  }, \yt{\  } \right)$$

$$\left(\varnothing,  
\yt{\ },  \yt{\ \\ \ }, \yt{\  & \  \\ \ }, \yt{\  & \ }, \yt{\  & \
\\ \ },\yt{\ & \ \\ \ & \ }, \yt{\ & \ & \ \\ \ & \ },  \yt{\ & \ &
\   \\ \  }, \yt{\  & \ & \   },  \yt{\  & \ & \  \\ \ },\yt{\  & \
\\ \ } \right)$$
 \caption{\textit{From top to bottom.} A vacillating tableau of length
   10, a hesitating tableau of length 8, an oscillating tableau of length
   11. 
In each case, the height is bounded by $2$.}
\label{fig:tableaux}
\mbox{}\\
{\color{gray!50} \hrule}
\end{figure}

\ytableausetup{textmode,nosmalltableaux}

\subsection{Lattice walks} 
Each integer partition represented as a Ferrers diagram in a tableau
sequence can also be represented by a vector of its parts. If the
tableau sequence is bounded by $k$, then a $k$-tuple is
sufficient. 

The sequence of vectors defines a lattice path. For example, each of the three tableau families above each directly
corresponds to a lattice path family in the region
\[W_k=\{(x_1,x_2, \dots, x_k): x_i \in \mathbb{Z}, x_1\geq x_2\geq \dots\geq x_k \geq 0 \}\] starting at the
origin $(0,\dots,0)$. We can explicitly define three classes of lattice
paths by translating the constraints on the tableau families. 

\noindent \textbf{Remark.} Twice in this article, in order to relate previous results, we use a translation of
this region and still identify it as $W_k$. The translated regions 
are identical to the original up to a small shift of coordinates.  This change is 
detailed explicitly in the text (the allowed sets of steps are never changed).

Let $e_i$ be the elementary basis vector with a 1 at position $i$ and 0
elsewhere. The steps in our lattice model are all elementary vectors,
with possibly one exception: the zero vector, also called \emph{stay
step}. The length of the walk increases with a stay step, but the
position does not change.
 
 \begin{description}
 \item[A $\boldsymbol{W_k}$-vacillating walk]  is a walk of
   \emph{even\/} length in $W_k$ using 
   (i) two consecutive stay steps; 
   (ii) a stay step followed by an $e_i$ step;
   (iii) a $-e_i$ step followed by a stay step; 
   (iv) a $-e_i$ step followed by an $e_j$ step.
 \item[A $\boldsymbol{W_k}$-hesitating walk] has even length and steps occur in the following pairs: (i) a stay step followed by an $e_i$  step; (ii) a
$-e_i$ step followed by a stay step; (iii) an $e_i$  step follow by $-e_j$ step. 
\item[A $\boldsymbol{W_k}$-oscillating walk]
 starts at the origin and
takes steps of type $e_i$ or $-e_i$, for $1\leq i\leq k$. It does not
permit stay steps. 
 \end{description}
Some examples are depicted in Figure~\ref{fig:walks}.

\begin{figure}
\begin{center}
\includegraphics[width=0.8\textwidth]{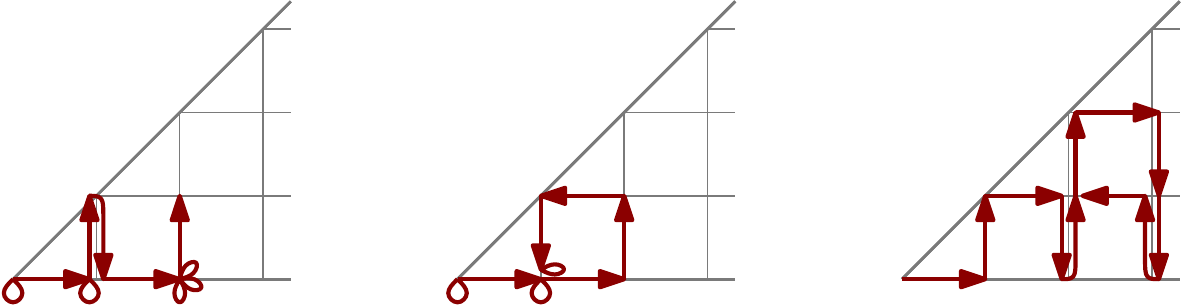}
\end{center}
\caption{\textit{From left to right.}  A $W_2$-vacillating walk, a
  $W_2$-hesitating walk, a $W_2$-oscillating walk. The stay steps are
  drawn as loops. (These walks correspond to the tableaux of {\sc Figure}~\ref{fig:tableaux}.) }
\label{fig:walks}
\mbox{}\\
{\color{gray!50} \hrule}
\end{figure}

\subsection{Open arc diagrams} Arc diagrams are a useful way to provide a
graphical representation of a combinatorial class. They are
particularly useful to detect certain patterns. Matchings and set
partitions are examples of classes that have natural representations
using arc diagrams. In the arc diagram representation of a set
partition of~$\{1, 2, \ldots, n\}$, a row of dots is
labelled from~$1$ to~$n$. A partition
block~$\{a_1, a_2, \ldots, a_j\}$, ordered~$a_1<a_2<\ldots<a_j$, is
represented by the set of arcs~$\{(a_1,a_2), (a_2,a_3), \ldots,(a_{j-1}, a_j)\}$
which are always drawn above the row of dots. 
We adopt the convention that a part of size one, say  $\{i\}$,
contributes a loop, that is a trivial arc $(i,i)$. In this work, we do
not draw the loops, although some authors do.  The set
partition~$\pi=\{\{1,3,7\},\{2,8\},\{4\},\{5,6\}\}$ is depicted as an
arc diagram in Figure~\ref{fig:expartition}. Matchings are represented
similarly, with each pair contributing an arc.
\begin{figure}[h!]
\centering
\begin{tikzpicture}[scale=0.4]
\foreach \x in {1,2,...,8}{
\node[pnt, label=below:{$\x$}](\x) at (\x,0){};}
\draw[bend left=45](1) to (3) to (7);
\draw[bend left=45](2) to (8);
\draw[bend left=45](5) to (6);
\end{tikzpicture}
\caption{The set partition~$\pi=\{1,3,7\},\{2,8\},\{4\},\{5,6\}$}
\label{fig:expartition}
\end{figure}
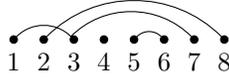

A set of~$k$ distinct arcs $(i_1, j_1), \dots, (i_k, j_k)$ forms a \emph{$k$-crossing} if
$i_1<i_2<\dots<i_k<j_1<j_2<\dots<j_k$. They form an \emph{enhanced $k$-crossing} if
$i_1<i_2<\dots<i_k\leq j_1<j_2<\dots<j_k$. (By convention, an isolated dot of the partition forms an enhanced $1$-crossing.)  They form a
\emph{$k$-nesting} if $i_1<i_2<\dots <i_k<j_k<\dots<j_2<j_1$. They
form an \emph{enhanced $k$-nesting} if $i_1<i_2<\dots
<i_k\leq j_k<\dots<j_2<j_1$ (As previously, $i_k=j_k$ means that $i_k$ is an isolated element in the set partition.).   Figure~\ref{fig:3cros}
illustrates a $3$-nesting, an enhanced $3$-nesting, and a
$3$-crossing. 


\begin{figure}[h!]
\center
\begin{tikzpicture}[scale=0.5]
\node[pnt] at (0,0)(1){};
\node[pnt] at (1,0)(2){};
\node[pnt] at (2,0)(3){};
\node[pnt] at (3,0)(4){};
\node[pnt] at (4,0)(5){};
\node[pnt] at (5,0)(6){};
\node[pnt] at (7,0)(1c){};
\node[pnt] at (8,0)(2c){};
\node[pnt] at (9,0)(3c){};
\node[pnt] at (10,0)(4c){};
\node[pnt] at (11,0)(5c){};
\node[pnt] at (13,0)(1b){};
\node[pnt] at (14,0)(2b){};
\node[pnt] at (15,0)(3b){};
\node[pnt] at (16,0)(4b){};
\node[pnt] at (17,0)(5b){};
\node[pnt] at (18,0)(6b){};
\node[pnt] at (20,0)(1d){};
\node[pnt] at (21,0)(2d){};
\node[pnt] at (22,0)(3d){};
\node[pnt] at (23,0)(4d){};
\node[pnt] at (24,0)(5d){};
\draw(1b)  to [bend left=45] (6b);
\draw(2b)  to [bend left=45] (5b);
\draw(3b)  to [bend left=45] (4b);
\draw(1d)  to [bend left=45] (5d);
\draw(2d)  to [bend left=45] (4d);
\draw(1)  to [bend left=45] (4);
\draw(2)  to [bend left=45] (5);
\draw(3)  to [bend left=45] (6);
\draw(1c)  to [bend left=45] (3c);
\draw(2c)  to [bend left=45] (4c);
\draw(3c)  to [bend left=45] (5c);
\end{tikzpicture}
\caption{Patterns in arc diagrams. \emph{From left to right:\/} a $3$-crossing, an enhanced $3$-crossing, a $3$-nesting, an
 enhanced $3$-nesting.}
\label{fig:3cros}
\end{figure}
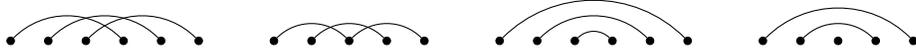

Recently, Burrill, Elizalde, Mishna and Yen~\cite{Buetal12}
generalized arc diagrams by permitting \emph{open arcs}: in these
diagrams each arc has a left endpoint but not necessarily a
right endpoint. The open arcs can be viewed as arcs ``under
construction". An \textit{open partition} (resp. an \emph{open
  matching}) is a set partition (resp. a matching) diagram with open
arcs. In open matchings, the left endpoint of an open arc is never the
right endpoint of another arc. Figure~\ref{fig:exopen} shows examples
of such diagrams.

\begin{figure}[h!]
\centering
\begin{tikzpicture}[scale=0.4]
\foreach \x in {1,2,...,9}{
\node[pnt, label=below:{$\x$}](a\x) at (\x,0){};}
\foreach \x in {1,2,...,10}{
\node[pnt, label=below:{$\x$}](\x) at (\x+12,0){};}

\draw(1)  to [bend left=45] (7);
\draw(3)  to [bend left=45] (9);
\draw(5)  to [bend left=45] (10);
\draw(4)  to [bend left=45] (6);

\draw[bend left=20](a3) to ++(1.5,2);
\draw[bend left=20](a5) to ++(1.5,2);
\draw[bend left=20](a7) to ++(1.5,2);
\draw[bend left=20](2) to ++(1.5,2);
\draw[bend left=20](8) to ++(1.5,2);

\draw[bend left=45](a2) to (a4);
\draw[bend left=45](a4) to (a5);
\draw[bend left=45](a1) to (a7);
\draw[bend left=45](a8) to (a9);
\end{tikzpicture}
\caption{An open partition and an open matching.}
\label{fig:exopen}
\end{figure}
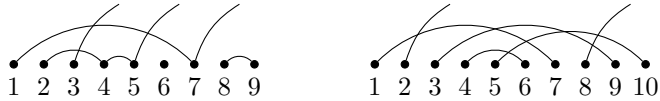 

We are also interested in crossing and nesting patterns in open
diagrams. Here we simplify the notation of~\cite{Buetal12}. A
\emph{$k$-crossing\/} in an open diagram is either a set of $k$
mutually crossing arcs (as before), or the union of $k-1$ mutually
crossing arcs and an open arc whose left endpoint is to the right
of the last left endpoint and to the left of the first right endpoint of the $k-1$
crossing arcs. A \emph{$k$-nesting\/} in an open diagram is either a
set of $k$ mutually nesting arcs, or a set of $k-1$ mutually nesting
arcs, and an open arc whose left endpoint is to the left of the $k-1$
nesting arcs. We generalize enhanced \mbox{$k$-crossings} and enhanced $k$-nestings
in an open diagram similarly. Examples are given in
Figure~\ref{fig:future}. If we want to point out that a crossing (or nesting) has no open
arc, we say that it is a \emph{plain} $k$-crossing (or $k$-nesting). 

\begin{figure}[h!]
\center
\begin{tikzpicture}[scale=0.5]
\node[pnt] at (0,0)(1){};
\node[pnt] at (1,0)(2){};
\node[pnt] at (2,0)(3){};
\node[pnt] at (3,0)(4){};
\node[pnt] at (4,0)(5){};

\node[pnt] at (7,0)(1c){};
\node[pnt] at (8,0)(2c){};
\node[pnt] at (9,0)(3c){};
\node[pnt] at (10,0)(4c){};
\node[pnt] at (13,0)(1b){};
\node[pnt] at (14,0)(2b){};
\node[pnt] at (15,0)(3b){};
\node[pnt] at (16,0)(4b){};
\node[pnt] at (17,0)(5b){};

\node[pnt] at (20,0)(1d){};
\node[pnt] at (21,0)(2d){};
\node[pnt] at (22,0)(3d){};
\node[pnt] at (23,0)(4d){};
\draw(1)  to [bend left=45] (4);
\draw(2)  to [bend left=45] (5);
\draw(3)  to [bend left=30] (3.5,1);
\draw[dotted]  (3.5,1)  to [bend left=30]  node [midway] {?} (5,0);

\draw(1c)  to [bend left=45] (3c);
\draw(2c)  to [bend left=45] (4c);
\draw(3c)  to [bend left=30] (10,0.5);
\draw[dotted]  (10,0.5)  to [bend left=30]  node [midway] {?} (11,0); 

\draw(1b)  to [bend left=25] (15.5,1.25);
\draw[dotted]  (15.5,1.25)  to [bend left=25]  node [midway] {?} (18,0); 
\draw(2b)  to [bend left=45] (5b);
\draw(3b)  to [bend left=45] (4b);

\draw(1d)  to [bend left=25] (22,1);
\draw[dotted]  (22,1)  to [bend left=25]  node [midway] {?} (24,0); 
\draw(2d)  to [bend left=45] (4d);

\end{tikzpicture}
\caption{Patterns in open diagrams. \emph{From left to right:\/} a $3$-crossing, an enhanced $3$-crossing, a $3$-nesting, and an
 enhanced $3$-nesting.   }
\label{fig:future}
\end{figure}
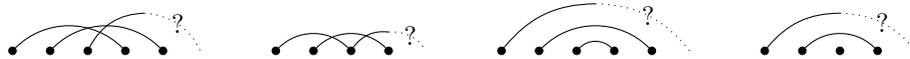

\section{Bijections}
\subsection{Description of Chen, Deng, Du, Stanley, Yan's bijection}
\label{sec:cddsy}
The work of Chen, Deng, Du, Stanley and Yan~\cite{Chetal07} describes
bijections between arc diagram families and tableau families. In this
section we summarize a selection of their results, and adapt it to our
needs.  Their main bijection maps a set partition~$\pi$ to a sequence
of Young tableaux\footnote{A \textit{Young tableau} is defined here as the filling of a Ferrers diagram with positive integers, such that the entries in each row and in each column are strictly decreasing (usually the entries are increasing; the reason for this change is explained later). The set of entries does not need to form an interval of the form $\left\{1,\dots,n\right\}$.}, the shapes of which form a vacillating
tableau, denoted by~$\phi(\pi)$. We do not describe the generalization of their construction to hesitating tableaux and oscillating tableaux (still due to Chen~\emph{et al.}), but it exists and it 
will be used for the proof of Propositions \ref{prop:om} and \ref{prop:ep}.

We describe here their bijection~$\phi$ but with a slight difference:
we read the arc diagrams from left to right, instead from right to
left as it was done originally. In concrete terms, it means that the
image of a partition~$\pi$ under $\phi$, as we write it, is the mirror
image of the actual $\phi(\pi)$.  Our approach is justified by the
fact that natural properties emerge when the reading direction is
swapped. This can be particularly seen through
Proposition~\ref{prop:below}, where the size of the crossings and
nestings around the $i$th dot is linked to the height and the width of
the $i$th Ferrers diagram.

Let $\pi$ be a set partition of size $n$. We are going to build from
$\pi$ a sequence of Young tableaux where the entries are decreasing in each row
  and each column -- the fact that we use decreasing order instead of
increasing order is a direct consequence of the change of the reading
direction.  The first entry is the empty Young tableau. We increment a
counter~$i$ by one from 1 to $n$. A given step in the algorithm
proceeds as follows.  If~$i$ is the right-hand endpoint of an arc
in~$\pi$, then delete~$i$ from the previous tableau (it turns out that
$i$ must be in a corner). Otherwise, replicate the previous
tableau. Then, after this move, if~$i$ is a left-hand endpoint of an
arc $(i,j)$ in $\pi$, insert~$j$ by the Robinson-Schensted-Knuth (RSK)
insertion algorithm for the decreasing order into the previous
tableau. If~$i$ is not a left-hand endpoint, replicate the previous
tableau.

The output of this process is a sequence of Young Tableaux starting
from and ending at the empty Young tableau. The sequence of shapes is
given by a vacillating tableau and is denoted~$\phi(\pi)$. \\

\ytableausetup{mathmode}
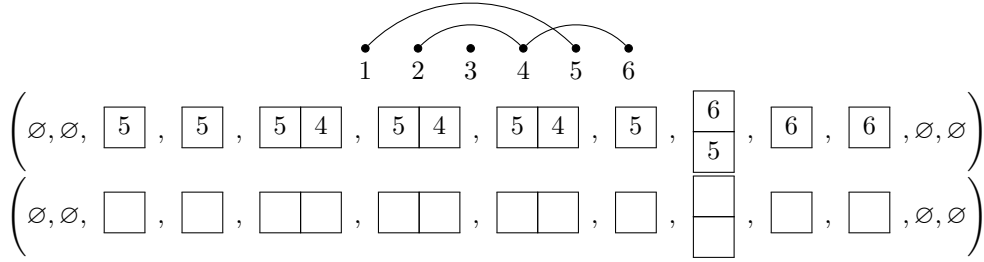
\begin{figure}
\center
\begin{tikzpicture}[scale=0.7]
\foreach \x in {1,2,...,6}{
\node[pnt, label=below:{$\x$}](\x) at (\x,0){};}
\draw(1)  to [bend left=45] (5);
\draw(2)  to [bend left=45] (4);
\draw(4)  to [bend left=45] (6);
\end{tikzpicture}

$\left(\varnothing, \varnothing,
\yt{5}, \yt{5}, \yt{5 & 4}, \yt{5 & 4}, \yt{5 & 4}, \yt{5}, \yt{6 \\ 5}, \yt 6, \yt 6, \varnothing, \varnothing\right)$

$\left(\varnothing, \varnothing,
\yt{\ }, \yt{\ }, \yt{\  & \ }, \yt{\  & \ }, \yt{\  & \ }, \yt{\ },  \yt{\  \\ \ }, \yt \ , \yt \ , \varnothing, \varnothing\right)$
\caption{\textit{Top.} The set partition $\pi =
  \{1,5\},\{2,4,6\},\{3\}$. \textit{Middle.} The corresponding Young
  tableau sequence. \textit{Bottom.} The vacillating tableau given
  by $\phi(\pi)$. }
\mbox{}\\
{\color{gray!50} \hrule}
\label{fig:chen}
\end{figure}

\newcommand{\ytt}[1]{\begin{ytableau}#1\end{ytableau}}
\ytableausetup{smalltableaux}

\noindent{\bf Example.} Consider the partition $\pi$ from Figure
  \ref{fig:chen}. The number $1$ is the left-hand endpoint of the arc
  $(1,5)$, but not the right-hand endpoint of any arc, so the first
  three Young tableaux are $\varnothing, \varnothing$, \ytt{5}. Similarly,
  $2$ is the left-hand endpoint of $(2,4)$ but not a right-hand
  endpoint, so the two following Young Tableaux are \ytt{5}, \ytt{5 &
    4}. The number $3$ is an isolated point, so the tableau \ytt{5 &
    4} is repeated twice. The number $4$ being the right-hand endpoint
  of $(2,4)$ and the left-hand endpoint of $(4,6)$, we delete $4$,
  then we add $6$: we obtain \ytt{5}, \ytt{6 \\ 5}. The rest of the
  sequence is given in Figure~\ref{fig:chen}.\\

Given a vacillating tableau
$(\varnothing,\lambda_1,\dots,\lambda_{2n-1},\varnothing)$, there
exists a unique way to fill the entries of the Ferrers diagrams into
Young tableaux so that it corresponds to an image of a set
partition. This has been proved in~\cite{Chetal07}, and implies that~$\phi$ is
a bijection.

In an arc diagram, we say that the \textit{segment $[i,i+1]$ is below}
a $k$-crossing if the arc diagram contains~$k$ arcs $(i_1, j_1), \dots, (i_k, j_k)$
such that $i_1<i_2<\dots<i_k\leq i$ and $i+1\leq
j_1<j_2<\dots<j_k$. Similarly, the \textit{segment $[i,i+1]$ is below}
a $k$-nesting if there exist $k$ arcs $(i_1, j_1), \dots, (i_k, j_k)$
such that $i_1<i_2<\dots <i_k\leq i$ and $i+1 \leq
j_k<\dots<j_2<j_1$. For instance, in Figure~\ref{fig:chen}, the
segment $[3,4]$ is below a $2$-nesting but not below a $2$-crossing,
while the segment $[4,5]$ is below a $2$-crossing but not below a
$2$-nesting.  With this definition we can formulate and prove a stronger version of \cite[Theorem 3.2]{Chetal07} (this property 
can also easily be seen in the growth diagram formulation of the bijection -- see~\cite{Krat06}).
\begin{prop} \label{prop:below} Let~$\pi$ be a partition of size $n$
  and $\phi(\pi)=(\lambda_0,\dots,\lambda_{2n})$. For
  every $i \in \{1,\dots,n\}$, the segment $[i,i+1]$ of $\pi$ is below a
  $k$-crossing (resp.~$k$-nesting) if and only if $\lambda_{2i}$ in $\phi(\pi)$ has
  at least~$k$ rows (resp.~$k$ columns).
\end{prop}

\noindent \textbf{Example.} We continue our example and verify
that $\lambda^{(6)} = \ytt{\ & \ }$ has $2$ columns but not $2$ rows,
and accordingly~$[3,4]$ is below a 2-nesting, but not a 2-crossing. 

\begin{proof} Let $(T_0,\dots,T_{2n})$ be the sequence of Young
  tableaux corresponding to the partition~$\pi$. We use some ingredients
  from the proof of Theorem 3.2 of
  \cite[p. 1562]{Chetal07}\footnote{Recall that one bijection is the
    mirror image of the other. So the indices differ between
    \cite{Chetal07} and here.}:
\begin{enumerate}
\item A pair $(i,j)$ is an arc in the representation of $\pi$ if and
  only if $j$ is an entry in $T_{2i},T_{2i+1},\dots,T_{2(j-1)}$; 
\item Let $\sigma_i = w_1 w_2\dots w_r$ denote the permutation of the
  entries of $T_i$ such that $w_1, w_2, \dots, w_r$ have been inserted
  in $(T_0,\dots,T_{2n})$ in this order;
\item The permutation $\sigma_i$ has an increasing subsequence of
  length $k$ if and only if the partition $\lambda_{i}$ has at least
  $k$ rows.
\end{enumerate}
The following statements are then  equivalent:
\begin{itemize}
\renewcommand\labelitemi{$\cdot$}
\item The segment $[i,i+1]$  is below a $k$-crossing. 
\item There exist $k$ arcs $(i_1, j_1), \dots, (i_k, j_k)$ in $\pi$ such that 
$$i_1<i_2<\dots<i_k\leq i\textrm{ and }i+1\leq j_1<j_2<\dots<j_k.$$ 
\item There exist $k$ numbers $j_1<j_2<\dots<j_k$ that are entries of  $T_{2i}$  such that $j_1, j_2, \dots,j_k$  have been inserted in this order in $(T_0,\dots,T_{2n})$.  
\item There exist $k$ numbers $j_1<j_2<\dots<j_k$ such that 
$j_1j_2 \dots j_k$ is a subsequence of $\sigma_{2i}$.  
\item The diagram $\lambda_{2i}$ has at least $k$ rows. 
\end{itemize}
The proof for $k$-crossings is similar.
\end{proof}

Considering all intervals $[i,i+1]$ for $1\leq i \leq n$, we recover the statement of Theorem 3.2 from \cite{Chetal07}.

\begin{corollary} 
\label{cor:classic} 
A set partition $\pi$ has no $(k+1)$-crossing (resp. no
$(k+1)$-nesting) if and only if no Ferrers diagram in the sequence
$\phi(\pi)$ has $k+1$ rows (resp. columns).
\end{corollary}

\noindent \textbf{Remark.} The \emph{crossing level\/} of a set
partition $\pi$, denoted $cr(\pi)$, is the maximal~$k$ such that $\pi$
has a $k$-crossing. Similarly, the \emph{nesting level\/} of a set
partition $\pi$, denoted $ne(\pi)$, is the maximal~$k$ such that $\pi$
has a $k$-nesting. Chen~\emph{et al.\/} conclude from the previous
corollary that the joint distribution of ~$cr$ and~$ne$ over all the
set partition diagrams of fixed size is symmetric. That is,
\[\sum_{\substack{\pi\textrm{ set partition diagram}\\\textrm{of size }n}   } x^{cr(\pi)} y^{ne(\pi)} = \sum_{\substack{\pi\textrm{ set partition diagram}\\\textrm{of size }n}   } y^{cr(\pi)} x^{ne(\pi)}.\]
Let $\tau$ denote transposition, the operation that transposes every Ferrers
diagram inside a vacillating tableau. Then $\phi^{-1} \circ \tau \circ
\phi$ swaps the crossing level and the nesting level of a set
partition. Moreover, note that $\phi^{-1} \circ \tau \circ \phi$
preserves the opener/closer sequence, i.e., if the number $i$ is an
isolated point (resp. a left endpoint, a right endpoint, a left and
right endpoint at the same time) in a partition $\pi$, then $i$ is an
isolated point (resp. a left endpoint, a right endpoint, a left and
right endpoint at the same time) in $\phi^{-1} \circ \tau \circ
\phi(\pi)$.
\subsection{Bijections with open partitions}
\label{ss:bijop}
Next we describe a generalization of the bijection of Chen~\emph{et
  al.} to the class of tableaux ending at a row shape. We thereby link
to the classes of Section~\ref{sec:combclasses}.

\begin{prop}
A bijection can be constructed between any two of the
following classes:
\begin{enumerate}
\item the set of open partition diagrams of length $n$ with no $(k+1)$-crossing, with $m$ open arcs;
\item the set of open partition diagrams of length $n$ with no $(k+1)$-nesting, with $m$ open arcs;
\item the set of vacillating tableaux of length~$2n$, with maximum
  height bounded by~$k$, ending in a row of length $m$;
\item the set of $W_k$-vacillating walks of length $2n$ ending
  at $(m, 0, \dots, 0)$.
\end{enumerate}
\label{prop:bij}
\end{prop}

\begin{proof} 
\textbf{Bijection (1) $\boldsymbol \Leftrightarrow$
    (3).} Since we would like to use  the aforementioned bijection $\phi$, the idea here simply consists in  closing every open arc in a very natural way -- we thereby recover classic closed diagrams. 
    
  Let us be more precise. Let~$\pi$ be an open partition diagram of
  length $n$ with $m$ open arcs and no $(k+1)$-crossing. We build a
  new partition diagram $\overline \pi$ of length $n+m$ without open
  arcs by closing the $m$ open arcs of $\pi$ in decreasing order. That
  is, if $i_1 < i_2 < \dots < i_m$ denote the positions of the $m$
  open arcs of $\pi$, the partition $\overline{\pi}$ is the closure,
  obtained by replacing the $m$ open arcs with the arcs
  $(i_1,n+m),(i_2,n+m-1),\dots,(i_m,n+1)$, as shown in
  Figure~\ref{fig:expartition2}. Note that the open arcs are closed in such a
  way that no new crossing is created.

\begin{figure}[h!]
\centering

\begin{tikzpicture}[scale=0.45]
\node[label=left:{$\pi=$}] at (1,0){};
\node[label=left:{$\overline{\pi}=$}] at (12,0){};
\foreach \x in {1,2,...,9}{
\node[pnt, label=below:{$\x$}](\x) at (\x,0){};
\node[pnt, label=below:{$\x$}](b\x) at (\x+11,0){};
}
\node[pnt, label=below:{\color{blue} $10$}, color=blue](e1) at (21,0){};
\node[pnt, label=below:{\color{blue} $11$}, color=blue](e2) at (22,0){};

\node[pnt, label=below:{\color{blue} $12$}, color=blue](e3) at (23,0){};
\draw[bend left=45](2) to (4);
\draw[bend left=45](4) to (5);
\draw[bend left=45](1) to (7);
\draw[bend left=45](8) to (9);
\draw[bend left=45, color=blue](3) to (9,2.1);
\draw[bend left=45, color=blue](5) to (9,1.4);
\draw[bend left=45, color=blue](7) to (9,0.7);
\draw[bend left=45](b2) to (b4);
\draw[bend left=45](b4) to (b5);
\draw[bend left=45](b1) to (b7);
\draw[bend left=45](b8) to (b9);
\draw[bend left=45,color=blue](b3) to (e3);
\draw[bend left=45,color=blue](b5) to (e2);
\draw[bend left=45,color=blue](b7) to (e1);

\end{tikzpicture}
\caption{\textit{Left.} An open partition diagram, $\pi$,  with $3$ open
  arcs. \textit{Right.} The corresponding closed partition diagram $\overline{\pi}$ ending with a $3$-nesting,
  obtained by closing the $3$ open arcs of $\pi$  in reverse order.}
\label{fig:expartition2}
\end{figure}
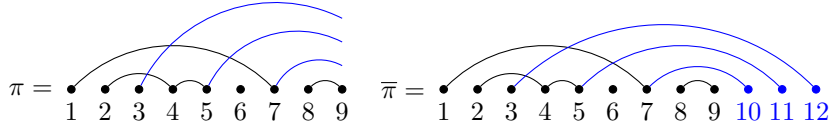

The $m$ last elements of $\overline{\pi}$ form the end of an
$m$-nesting. Consequently, each crossing of $\overline{\pi}$ has at most one element
inside \mbox{$\{n+1,\dots,n+m\}$}; so the preimage of any
$\ell$-crossing of $\overline{\pi}$ is also an~$\ell$-crossing. As~$\pi$ has no
$(k+1)$-crossing, the diagram $\overline{\pi}$ has no~$(k+1)$-crossing.

Let $\phi(\overline{\pi}) = (\lambda_0,\dots,\lambda_{2(n+m)})$ be the image of
$\overline{\pi}$ under $\phi$. By Corollary~\ref{cor:classic}, the height of
this vacillating tableau is bounded by~$k$. Moreover, the segment
$[n,n+1]$ in $\overline{\pi}$ is below an $m$-nesting but not below a
$2$-crossing. By Proposition~\ref{prop:below}, it means that
$\lambda_{2n}$ is a column with at least $m$ rows. Since $\phi(\overline{\pi})$ ends with an empty diagram and one can delete at most one cell every two steps,
$\lambda_{2n}$ has exactly $m$ rows. Thus,
$(\lambda_0,\dots,\lambda_{2n})$ is a vacillating tableau of
length~$2n$, with maximum height bounded by~$k$, ending in a column of
length $m$.
  
The transformation is bijective: a vacillating
tableau~$(\lambda_0,\dots,\lambda_{2n})$ from the set~(3) can be
concatenated with
$((m-2), (m-2), \dots, (1),\varnothing, \varnothing)$, where $(j)$
denotes the partition of $i$ only composed of a single part of size
$j$. If we change its preimage under $\phi$ by opening the arcs ending in
$\{n+1,\dots,n+m\}$ into~$m$ open arcs, we recover
the initial open diagram $\pi$.
  
\textbf{Bijection (2) $\boldsymbol \Leftrightarrow$ (3).}  The
previous bijection is adapted with an additional application of the
transposition operator $\tau$.

\textbf{Bijection (3) $\boldsymbol \Leftrightarrow$ (4).} This is a
straightforward consequence of the encoding. As the vacillating tableaux end at a row of length~$m$,
the endpoints of the walks must be the point $(m,0,\dots,0)$
\end{proof}

The open diagram case inherits many properties from the closed diagram
case. For example, the statistics of crossing level and nesting level
are equidistributed. Also, the problem of finding a direct bijection
between open partitions with no $k$-crossing and open partitions with
no $k$-nesting without going through the vacillating tableaux seems to
be as difficult as the closed case. 

However,  the nesting level and the crossing level do not
have symmetric joint distribution for open partitions. This constitutes a
difference with the (closed) partition diagrams. 

Furthermore, the other generalizations of Chen~\emph{et al.}  --
specifically the ones that concern the hesitating and oscillating
tableaux -- can also be extended to tableaux ending at a row shape,
and open partitions. The proofs are similar.

\begin{prop} \label{prop:om}
The following classes are in bijection:
\begin{enumerate}
\item the set of open matching diagrams of length $n$ with no~$(k+1)$-crossing, with $m$ open arcs;
\item the set of open matching diagrams of length $n$ with no~$(k+1)$-nesting, with $m$ open arcs;
\item the set of oscillating tableaux of length~$n$, with height bounded by~$k$, ending in
  a row of length $m$;
\item the set of $W_k$-oscillating walks of length $n$ ending
  at $(m, 0, \dots, 0)$.
\end{enumerate}
\end{prop}

\begin{prop} \label{prop:ep}
The following classes are in bijection:
\begin{enumerate}
\item the set of open partition diagrams of length $n$ with no enhanced ~$(k+1)$-crossing, with $m$ open arcs;
\item the set of open partition diagrams of length $n$ with no
  enhanced~$(k+1)$-nesting, with $m$ open arcs;
\item the set of hesitating tableaux of length~$2n$, with  height bounded by~$k$, ending in
  a row of length $m$;
\item the set of $W_k$-hesitating walks of length $2n$ ending
  at $(m, 0, \dots, 0)$.
\end{enumerate}
\end{prop}

\section{Young tableaux, involutions and open matchings}
\label{sec:YTBH}
\subsection{Bijections}
\label{sec:YT-biject}

We can now prove our first main result, namely  Theorem~\ref{thm:SYT}. Our
strategy is to use Proposition~\ref{prop:om}, and prove the following
result, from which Theorem~\ref{thm:SYT} is a straightforward
consequence.

\begin{proposition}
\label{thm:VACC}
The set of standard Young tableaux of size
$n$ with height bounded by $2k$ and $m$ odd columns 
are in bijection with the set of open
matching diagrams of length $n$,  with $m$ open arcs and with no
$(k+1)$-crossing.
\end{proposition}

As far as we can tell, this theorem was first conjectured by
Burrill~\cite{Burr14}\footnote{More precisely, this conjecture used
  open matchings with no~$(k+1)$-nesting.}.  Our proof uses the
Robinson-Schensted-Knuth (RSK) correspondence, and the bijection of
Chen~\emph{et al.}.

A different proof was communicated to us by Christian
Krattenthaler~\cite{Krat15}.  It relies on the RSK correspondence like
our proof, but also on \textit{jeu de taquin} (an operation on Young
tableaux invented by Sch\"utzenberger~\cite{Schu76}). We note that the two bijections
differ: our bijection has the
advantage of preserving -- just like the Chen \emph{et al.\/} construction --
the ``opener/closer" sequence (in a formulation using diagrams on both
sides of the bijection; cf Lemma \ref{lem:psi} for more details), a strong property which does not clearly appear in
 Krattenthaler's alternative. His proof passes through growth
diagrams~\cite{Krat06}.

The following lemma presents a classic property of the RSK correspondence.

\begin{lemma} (Robinson-Schensted-Knuth correspondence)
\label{lem:rsk}
The set of standard Young tableaux of size
$n$ with height bounded by $k$ and $m$ odd columns is in bijection with involutions of size $n$ with
$m$ fixed points and no decreasing subsequence of length $k+1$.
\end{lemma}

\ytableausetup{nosmalltableaux}
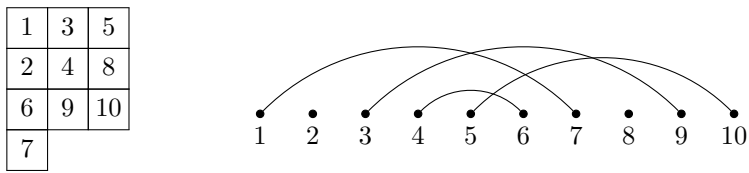
\begin{figure}[h!]
$\yt{1 & 3 & 5 \\ 2 & 4 & 8 \\ 6 & 9 & 10 \\ 7}$\quad \quad \quad
\begin{minipage}{0.45 \textwidth}
\center
\begin{tikzpicture}[scale = 0.7]
\foreach \x in {1,2,...,10}{
\node[pnt, label=below:{$\x$}](\x) at (\x,0){};}
\draw(1)  to [bend left=45] (7);
\draw(3)  to [bend left=45] (9);
\draw(5)  to [bend left=45] (10);
\draw(4)  to [bend left=45] (6);
\end{tikzpicture}
\end{minipage}

\caption{%
\textit{Left.\/} A standard Young tableau $Y$ of size $10$.
\textit{Right.\/} The arc diagram representation of the involution $(1\ 7)(3\ 9)(4\ 6)(5\ 10)$. This involution is the image of $(Y,Y)$ under the RSK correspondence.}
\label{fig:correspondance}
\end{figure}

As a first step, Lemma~\ref{lem:rsk}  yields combinatorial objects that
are close to open matchings. Indeed, involutions have a very
natural arc diagram representation: cycles $(i\,j)$ are represented by
an arc, and fixed points are isolated dots. An example is shown in
Figure~\ref{fig:correspondance}.  We can map involutions into the set
of open matchings by simply changing every isolated point into an open
arc. Under this map, there is a simple correspondence between
decreasing sequences in an involution and nestings in the open
diagram.

\begin{lemma} 
\label{l:involutionnesting}
Let $k \in \mathbb Z_{\geq 1}$. An involution has no decreasing
subsequence of length $2k+1$ if and only if there is no enhanced
$k$-nesting in its arc diagram representation.
\end{lemma}

\begin{proof} Let~$\alpha$ be an involution. If its arc diagram
  has an enhanced $k$-nesting then $\alpha$ contains $k$
  cycles $(i_1\, j_1),\dots,(i_k\,j_k)$ that satisfy
  $i_1 < i_2 < \dots < i_{k} \leq j_{k} < \dots < j_1$, which clearly induces a decreasing subsequence
  of length $2k-1$.
  
  Conversely, assume that there exist $2k-1$ numbers
  $i_1 < i_2 < \dots < i_{2k-1}$ such that
  $ \alpha(i_{2k-1}) < \dots < \alpha(i_1)$. If
  $\alpha(i_{k}) - i_k \geq 0$, then
  $i_1 < \dots < i_k \leq \alpha(i_k) < \dots < \alpha(i_1)$: this
  means that $(i_1,\alpha(i_1)),\dots,(i_k,\alpha(i_k))$ form an
  enhanced $k$-nesting. Otherwise, $\alpha(i_{k}) - i_{k} \leq 0$.
  Thus
  $\alpha(i_{2k-1}) < \dots < \alpha(i_k) \leq i_k < \dots <
  i_{2k-1}$:
  the arcs $(\alpha(i_{2k-1}),i_{2k-1}),\dots,(\alpha(i_k),i_k)$ form
  an enhanced $k$-nesting. \end{proof}

By the two preceding lemmas, the proof of Proposition 8 is reduced to
the proof that involution diagrams of length $n$ with $m$ fixed points and
no enhanced $(k+1)$-nesting are in bijection with open matching diagrams
of length $n$ with $m$ open arcs and no $(k+1)$-crossing. This is
established by the following lemma.

\begin{lemma} There is a bijection $\psi$ between involution diagrams  and open matching diagrams, such that for $\alpha$ an involution 
diagram and $\beta=\psi(\alpha)$, the diagrams $\alpha$ and $\beta$ have same length, the number of fixed points in $\alpha$ is the number
of open arcs in $\beta$, and for any $\ell\geq 1$ there is an enhanced $\ell$-nesting in $\alpha$ if and only if there is an $\ell$-crossing
in $\beta$. In addition the opener/closer sequence of $\alpha$ (seeing fixed points as openers) is the same as the opener/closer sequence of $\beta$. \label{lem:psi}
\end{lemma}





\begin{proof}[Proof of Proposition~\ref{thm:VACC}] 
  We describe~$\psi$, a bijective map between involutions and open
  matchings. It is formed as a composition of other maps. We have
  already defined $\phi$, the bijection from set partition diagrams to
  vacillating tableaux from Section~\ref{sec:cddsy}, and~$\tau$, the
  transpose action which can be applied to any tableau sequence. We
  add~$\iota$, the operation that changes every isolated dot in an
  involution diagram into an open arc. Let $\psi$ be the composition
  $\iota \circ \phi^{-1} \circ \tau \circ \phi$. Figure~\ref{fig:psi}
  shows an example of the action of~$\psi$. 

\begin{figure}
\begin{center}
\includegraphics[width=0.8\textwidth]{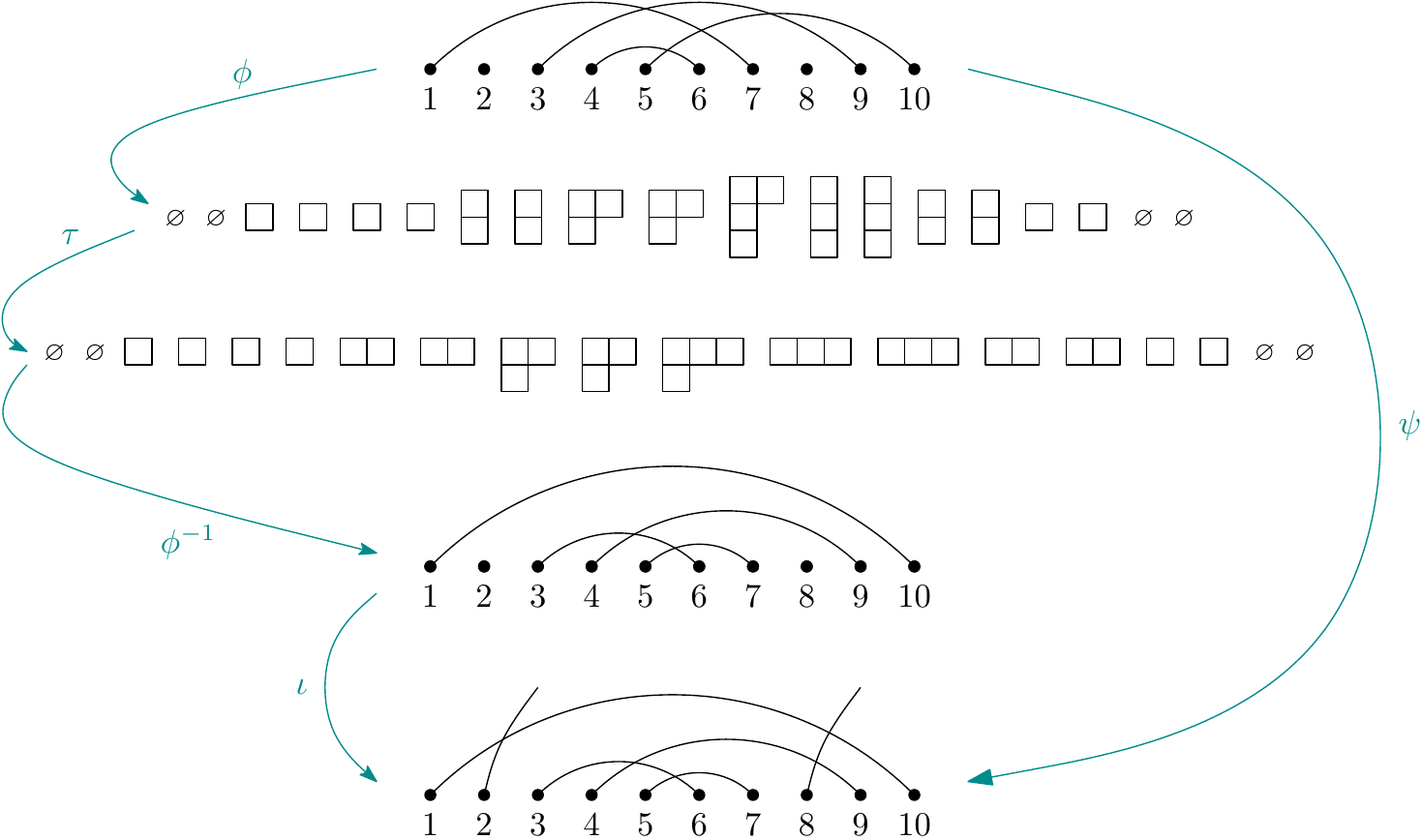}
\end{center}
\caption{Image of the involution of Figure~\ref{fig:correspondance}  under $\psi$. }
\label{fig:psi}
\mbox{}\\
{\color{gray!50} \hrule}
\end{figure}

Since $\phi$, $\tau$ and $\iota$ can all be reversed,
the mapping~$\psi$ is bijective. Moreover, recall from the end of Section~\ref{sec:cddsy}, the
mapping $\phi^{-1} \circ \tau \circ \phi$ preserves the opener-closer
sequence. Therefore, every involution of size~$n$ with~$m$ fixed
points is mapped under $\psi$ to an open matching diagram of size $n$
with $m$ open arcs. 

Assume that an involution~$\alpha$ has an enhanced $\ell$-nesting
$(i_1,j_1)\dots(i_\ell,j_\ell)$. If $i_\ell \neq j_\ell$, this enhanced nesting is
also a plain~$\ell$-nesting. By the remark at the end of
Section~\ref{sec:cddsy}, we know that
$\phi^{-1} \circ \tau \circ \phi(\alpha)$
has a $\ell$-crossing, so the same holds for $\psi(\alpha)$.

If $i_\ell = j_\ell$, then $i_\ell$ is a fixed point of $\alpha$ and hence an
open arc in $\psi(\alpha)$. Moreover, the segment $[i_\ell,i_\ell+1]$ is
below the $(\ell-1)$-nesting $(i_1,j_1)\dots(i_{\ell-1},j_{\ell-1})$. So by
Proposition~\ref{prop:below}, the $2i_\ell$th diagram of $\phi(\alpha)$
has at least~$\ell$ columns. The $2i_\ell$th diagram of
$\tau \circ \phi(\alpha)$ has then $k$ rows, and so $[i_\ell,i_\ell+1]$ is
below a $(\ell-1)$-crossing in $\psi(\alpha)$. Thus, $i_\ell$ is in
$\psi(\alpha)$ an open arc below a $(\ell-1)$-crossing: the open matching
$\psi(\alpha)$ has a~$\ell$-crossing. The converse is proved similarly.

In summary,~$\psi$ is a bijection between involution diagrams of size
$n$ with $m$ fixed points and no enhanced~$k$-nesting and open
matchings diagrams of size~$n$ with~$m$ open arcs and no
$k$-crossing.
\end{proof}

The standard Young tableau in Figure~\ref{fig:correspondance} is
mapped to the open arc diagram at the bottom of
Figure~\ref{fig:psi}. Here, the parameter $m$ takes value 2.

Remark that standard Young tableaux with height bounded by~odd numbers
are also characterized in terms of open matching diagrams (but this time constrained by the plain nestings or crossings).

\begin{prop} 
  \label{prop:yt}
The following classes are in bijection:
\begin{enumerate}
\item[(i)] the set of standard Young tableaux of size $n$ with $m$ odd columns and height bounded by~$2k-1$;
\item[(ii)] the set of involutions of size $n$ with $m$ fixed points and no decreasing subsequence of length $2k$;
\item[(iii)] the set of open matching diagrams of length $n$ with no \emph{plain}~$k$-crossing and with $m$ open arcs;
\item[(iv)] the set of open matching diagrams of length $n$ with no \emph{plain}~$k$-nesting and with $m$ open arcs.
\end{enumerate}
\end{prop}

\begin{proof} The RSK correspondence (specifically, the property described in
  Lemma~\ref{lem:rsk}) gives a straightforward bijection between (i)
  and (ii). Then, seeing isolated points as open arcs, it is easy to
  adapt Lemma \ref{l:involutionnesting} in order to show the
  correspondence between (ii) and (iv). Finally the bijection between
  (iii) and (iv) is given by $\phi^{-1} \circ \tau \circ \phi$ (still by
  interpreting isolated points as open arcs), where $\phi$ and $\tau$ are
  defined in Section~\ref{sec:cddsy}.
 \end{proof}


\subsection{A new symmetric joint distribution for involutions}\mbox{}
While looking for the previous bijection we found a
surprising symmetry property for involutions, which is now presented.
Section~\ref{sec:cddsy} contained the definition of \textit{nesting
level}; in the context of involution diagrams, the notion can be refined in two different ways, 
depending on whether we regard involution diagrams as enhanced set partition diagrams or as open matchings.

 The \emph{enhanced
  nesting level\/} of an involution $\alpha$, denoted
$\enhne(\alpha)$, is the maximal number of \emph{dots} in an enhanced
nesting of~$\alpha$ (note that $k$ marks the number of dots not
number of arcs). 
Pursuant to Lemma \ref{l:involutionnesting}, the number
$\enhne(\alpha)$ is also the length of the longest decreasing
subsequence of~$\alpha$.  Similarly, we define the \emph{open nesting
  level\/} of an involution $\alpha$, denoted $\futne(\alpha)$: after
transforming the diagram of $\alpha$ into an open matching by changing
every isolated point into an open arc, the open nesting level of
$\alpha$ is the maximal number of dots inside a nesting.

Remark that an enhanced nesting and a nesting in an open diagram are
identical if they both have an even number of dots; these are then
plain nestings. The difference is made when the number of dots is odd,
say $2k+1$. In this case, an enhanced nesting is made of a dot
\emph{below} a plain $k$-nesting, while a nesting in an open matching
is made of an open arc to the \textit{left} of a plain
$k$-nesting. This justifies the notation $\enhne$ and
$\futne$. Figure~\ref{fig:enhfut} compares the two patterns.

\begin{example}
The open nesting level of the involution
$(1\ 7)(3\ 9)(4\ 6)(5\ 10)$, depicted in Figure
\ref{fig:correspondance}, is $5$: the numbers $2, 3, 4, 6, 9$ form a
nesting if we transform the dot $2$ into an open arc. However the enhanced nesting level of the same involution
is~$4$: there is no dot below any $2$-nesting. 
\end{example}

\begin{figure}[h!]

\begin{minipage}{0.33 \textwidth}
\center
\begin{tikzpicture}[scale = 0.5]
\foreach \x in {1,2,...,5}{
\node[pnt](\x) at (\x,0){};}
\draw(1)  to [bend left=45] (5);
\draw(2)  to [bend left=45] (4);
\end{tikzpicture} \\[12pt]
\begin{tikzpicture}[scale = 0.5]
\foreach \x in {1,2,...,6}{
\node[pnt](\x) at (\x,0){};}
\draw(1)  to [bend left=45] (6);
\draw(2)  to [bend left=45] (5);
\draw(3)  to [bend left=45] (4);
\end{tikzpicture}
\end{minipage} 
\begin{minipage}{0.33 \textwidth}
\center
\begin{tikzpicture}[scale = 0.5]
\foreach \x in {1,2,...,5}{
\node[pnt](\x) at (\x,0){};}
\draw(2)  to [bend left=45] (5);
\draw(3)  to [bend left=45] (4);
\draw[bend left=25,color=gray] (1)  to  (5,1);
\end{tikzpicture} \\[12pt]
\begin{tikzpicture}[scale = 0.5]
\foreach \x in {1,2,...,6}{
\node[pnt](\x) at (\x,0){};}
\draw(1)  to [bend left=45] (6);
\draw(2)  to [bend left=45] (5);
\draw(3)  to [bend left=45] (4);
\end{tikzpicture} 
\end{minipage} 

\caption{\textit{Left.} An enhanced $3$-nesting with respectively $5$
  and $6$ dots.  \textit{Right.}  A  $3$-nesting with respectively $5$
  and $6$ dots (in an open matching). }
\label{fig:enhfut}
\end{figure}
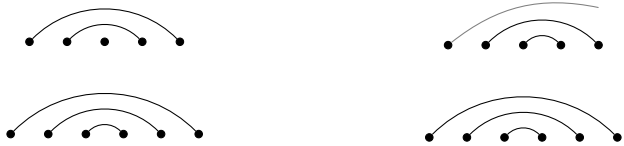

A (weak) link between the two statistics can be easily derived from the preceding study, as stated in the following proposition.

\begin{prop} There is a bijection $\theta$ from involution diagrams to involution diagrams such that for $\alpha$ any involution 
diagram, $\beta=\theta(\alpha)$, and $\ell \geq 1$, there is an enhanced
$\ell$-nesting in $\alpha$ if and only if there is an $\ell$-nesting 
in the open matching obtained by changing every fixed point in $\beta$ by an open arc.
In other words,
there exists a bijection $\theta$ between involutions $\alpha$
such that $\enhne(\alpha)=2k-1$ or $2k$ and involutions $\beta$ such
that $\futne(\beta)= 2k-1$ or $2k$.

In addition,  $\theta$ preserves the length, the  number of fixed points and the  opener/closer sequence (viewing fixed points as openers).
\label{prop:theta}
\end{prop}

\begin{proof} We define $\theta$ as
the composition of
\begin{enumerate}
\item the mapping $\phi$ described by Lemma \ref{lem:psi};
\item the operation that changes isolated points into open arcs; 
\item the mapping from open matchings to oscillating tableaux with bounded width ending a column (see Subsection~\ref{ss:bijop} and more precisely Proposition \ref{prop:om});  
\item the transposition of the Ferrers diagrams;
\item the mapping from oscillating tableaux with bounded height ending at a column to open matchings;
\item the operation that changes open arcs into isolated points.
\end{enumerate}
All the properties of $\theta$ presented in the statement of this proposition are direct consequences of  Lemma \ref{lem:psi} and Proposition \ref{prop:om}.
\end{proof}

Note that an enhanced nesting is preserved when the diagram is
reflected. This is not true for an odd nesting in an open diagram, because the isolated
point must be to the left of the nesting. Despite the fact they do not
share this property, the enhanced nesting level and the open nesting level
have symmetric distribution, as stated in the following theorem.

\begin{theorem} The statistics $\futne$ and $\enhne$ have a symmetric joint distribution over all the involutions of size $n$ with $m$ fixed points, i.e.,
$$\sum_{\substack{\alpha\textrm{ involution}\\\textrm{of size }n\\\textrm{with }m\textrm{ fixed points}}} x^{\futne(\alpha)} y^{\enhne(\alpha)} = \sum_{\substack{\alpha\textrm{ involution}\\\textrm{of size }n\\\textrm{with }m\textrm{ fixed points}}} y^{\futne(\alpha)} x^{\enhne(\alpha)}.$$
\label{thm:sym}
\end{theorem}


\noindent \textbf{Remark.} The bijection $\theta$ from Proposition \ref{prop:theta} does not swap the statistics
 $\enhne$ and $\futne$. For instance, the involution
 $\alpha_1 = (1\ 5)(2\ 3)$ of size $5$ is mapped to the
 involution $\alpha_2 = (2\ 3)(4\ 5)$: we have
 $\enhne(\alpha_1)=\futne(\alpha_1) = 4$ but $\futne(\alpha_2)= 3$ and
 (even worse!) $\enhne(\alpha_2)=2$. Nonetheless, the existence of the
 function $\theta$ (and more particularly the fact that an involution with
 enhanced nesting level $2k-1$ or $2k$ is mapped under $\theta$ to an
 involution with open nesting level $2k-1$ or $2k$) is sufficient to prove
 Theorem~\ref{thm:sym}.

 \begin{proof} Consider all involutions of fixed size, with a fixed
   number of fixed points.  Let $a_{i,j}$ be the number of involutions
   $\alpha$ in this class of such that $\enhne(\alpha) = i$ and
   $\futne(\alpha) = j$. By Proposition \ref{prop:theta}, the bijection $\theta$ maps
   involutions~$\alpha$ such that $\enhne(\alpha)=2k-1$ or $2k$ to
   involutions~$\beta$ such that $\futne(\beta)= 2k-1$ or $2k$; hence
\begin{equation}\sum_{j \geq 0} a_{2k-1,j} +  a_{2k,j} = \sum_{i \geq 0} a_{i,2k-1} +  a_{i,2k}. \label{eq:a}
\end{equation}
We can simplify the expression in Equation~\eqref{eq:a} 
 as the values~$\enhne(\alpha)$ and~$\futne(\alpha)$
can only differ by at most one for a given involution $\alpha$.
Indeed, if $\ell$ denotes the maximal number of arcs inside a nesting
of an involution, the open nesting level and the enhanced nesting
level must equal either $2\ell$ or $2 \ell +1$. Therefore, $a_{i,j}=0$
except for pairs $(i,j)$ of the form $(2\ell,2\ell)$,
$(2\ell,2\ell+1)$, $(2\ell+1,2\ell)$ or $(2\ell+1,2\ell+1)$. Equation
\eqref{eq:a} can be thus rewritten as:
$$ a_{2k-1,2k-1} + a_{2k-1,2k-2} +  a_{2k,2k} +  a_{2k,2k+1}
=
a_{2k-2,2k-1} + a_{2k-1,2k-1} +  a_{2k,2k} +  a_{2k+1,2k}, 
$$
or after simplification
$$ a_{2k,2k+1} - a_{2k+1,2k} = a_{2k-2,2k-1} -  a_{2k-1,2k-2}.$$
In other words, the sequence $\left(a_{2k,2k+1} - a_{2k+1,2k}\right)$ is constant over all $k \geq 0$. But since it equals $0$ for $k=0$, we have for every $k \geq 0$,
$$a_{2k,2k+1} = a_{2k+1,2k}.$$
The other terms $a_{i,j}$ such that $i \neq j$ vanish, so the last
equality is sufficient to conclude the proof.
\end{proof}
 
The previous proof is simple but not constructive: can we describe an
involution (on involutions) that swaps the statistics $\futne$ and
$\enhne$? The answer is yes, and a description can be given in terms
of iterations of $\theta$, where $\theta$ is the mapping defined by Proposition \ref{prop:theta}.

\begin{lemma} Let $\theta^{(\ell)}$ be the $\ell$th iteration of
  $\theta$ and $A_{i,j}$ be the set of involutions $\alpha$ such that
  $\enhne(\alpha) = i$ and $\futne(\alpha) = j$.

  For every $\alpha$ in $A_{2k,2k+1}$ with $k \geq 0$, there exists $m
  \geq 1$ such that \mbox{$\theta^{(\ell)}(\alpha) \notin A_{2k,2k+1}
    \cup A_{2k+1,2k}$} for \mbox{$\ell \in \{1,\dots,m-1\}$}, and
  $\theta^{(m)}(\alpha) \in A_{2k+1,2k}$. Moreover, for every
  $\alpha'$ in $A_{2k+1,2k}$, there exists $m' \geq 1$ such that
  $\theta^{(\ell')}(\alpha') \notin A_{2k,2k+1} \cup A_{2k+1,2k}$
  for~$\ell' \in \{1,\dots,m'-1\}$, and $\theta^{(m')}(\alpha') \in
  A_{2k,2k+1}$.

  In other words, in the orbit of any involution under $\theta$ (this
  orbit is cyclic since $\theta$ is bijective and the set of
  involutions of fixed size is finite), the elements of
  $A_{2k,2k+1} \cup A_{2k+1,2k}$ alternate between $A_{2k,2k+1}$ and
  $A_{2k+1,2k}$.
\end{lemma}
An example of this correspondence is illustrated in
Figure~\ref{fig:schema}.

\begin{figure}
\begin{center}
\includegraphics[width=0.8\textwidth]{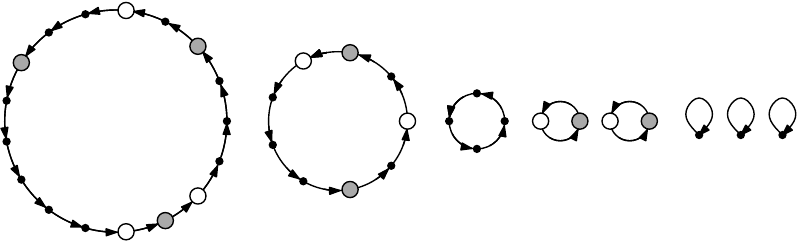}
\end{center}
\caption{Schematic representation of typical orbits under $\theta$.
  The white circles represent the elements of $A_{2k,2k+1}$, the gray
  circles the elements of  $A_{2k+1,2k}$ and the small points are
  the remaining elements.}
\label{fig:schema}
\mbox{}\\
{\color{gray!50} \hrule}
\end{figure}

\begin{proof} Consider $i > 0$ such that
  $\theta^{(i)}(\alpha) = \alpha$ (such an $i$ exists as
  $\theta$ acts bijectively on the finite set of involutions of a
  fixed length). Since $\futne(\alpha)= 2k+1$, we
  have $\enhne(\theta^{(i-1)}(\alpha))>2k$.

Let $m$ denote the smallest $j > 0$ such $\enhne(\theta^{(j)}(\alpha))>2k$. We have then $\enhne(\theta^{(m-1)}(\alpha))\leq 2k$. 
Using the properties of $\theta$ we know that $\futne(\theta^{(m)}(\alpha))\leq 2k$, hence $\enhne(\theta^{(m)}(\alpha))\leq 2k+1$. As  $\enhne(\theta^{(m)}(\alpha))>2k$, we must have $\enhne(\theta^{(m)}(\alpha)) = 2k+1$ and $\futne(\theta^{(m)}(\alpha)) = 2k$.

We have just showed that for $\alpha$ and $\theta^{(i)}(\alpha)$
belonging to $A_{2k,2k+1}$, there exists $m \in \{1,\dots,i-1\}$ such
that $\theta^{(m)}(\alpha) \in A_{2k+1,2k}$. The proof is over if we
manage to show the statement concerning $\alpha'$. This can be done either by using the same reasoning or by an
argument of cardinality (with Theorem~\ref{thm:sym}).
\end{proof}

The previous lemma sets out how to build the desired
involution. Essentially, from an involution of $A_{2k+1,2k}$, we
iterate $\theta$ until obtaining an involution of $A_{2k,2k+1}$. If,
on the other hand, the involution belongs to $A_{2k,2k+1}$ we want to
go backward, so we iterate $\theta^{-1}$ until obtaining an involution
of $A_{2k+1,2k}$. If an involution does not fall under one of the
previous forms, it necessarily belongs to a set of the form
$A_{\ell,\ell}$, and we can then set this involution as a fixed point.

\begin{prop} Let $\alpha$ be an involution and $A_{i,j}$ be the set of
  involutions $\alpha$ such that $\enhne(\alpha) = i$ and
  $\futne(\alpha) = j$. If $\alpha \in A_{2k,2k+1}$, set $m_\alpha$ as
  the smallest integer $m$ such that
  $\theta^{(m)}(\alpha) \in A_{2k+1,2k}$. If $\alpha \in A_{2k+1,2k}$,
  set $m_\alpha$ as the opposite of smallest integer $m$ such that
  $\theta^{(-m)}(\alpha) \in A_{2k,2k+1}$. Otherwise, set $m_\alpha$
  as $0$.

  The mapping $\alpha \mapsto \theta^{(m_\alpha)}(\alpha)$ is an
  involution on the class of involutions that exchanges the statistics
  $\enhne$ and $\futne$. It preserves the size of the involutions, the
  number of fixed points and the opener/closer sequence\footnote{if we
    consider isolated points as left endpoints of arcs}.
\end{prop}

\noindent \textbf{Remark.} What about the \textit{open crossing level} of
an involution, that is to say the maximum number of dots contained
in a~$k$-crossing, when this involution is transformed into an open matching? It is easy to see that the open crossing level shares
the same distribution as the open nesting level or the enhanced nesting
level (in particular via the bijection $\phi^{-1} \circ \tau \circ
\phi$). However, this statistic does not have a symmetric
distribution, whether it is with the open nesting level or with the
enhanced nesting level.

\subsection{Generating function expressions}
\label{sec:sytbh-egf}
One consequence of Theorem~\ref{thm:SYT} is a collection of new generating
function expressions for standard Young tableaux of bounded
height. They come from an application of 
enumeration results of Weyl chamber walks~\cite{GeZe92,GrMa93}.

\subsubsection{A Determinant Expression}
The generating functions for Young tableaux of bounded height and Weyl
chamber walks can both be expressed in terms of Bessel functions. We
denote by \[b_j(x)=I_j(2x)=\sum_n\frac{(2x)^{2n+j}}{n!(n+j)!},\] the
\emph{hyperbolic Bessel function of the first kind of order~$j$}.

Let $\widetilde{Y}_{k}(t)$ be the exponential generating function for
the class of standard Young tableaux with height bounded by
$k$. Formulas for $\widetilde{Y}_k(t)$ follow from works of Gordon,
Houten, Bender and Knuth~\cite{Gord71, GoHo68, BeKn72}, which depend
on the parity of~$k$. We are only interested in the even values here.
\begin{theorem}[\cite{Gord71, GoHo68, BeKn72}]
The exponential generating function for the class of standard Young
tableaux of height bounded by~$2k$ is given by
\[
\widetilde{Y}_{2k}(t)= \det[b_{i-j}(t)+b_{i+j-1}(t)]_{1\leq i,j\leq k}.
\]
\end{theorem}

Around the same time, Grabiner-Magyar~\cite{GrMa93} determined a formula for the
exponential generating function of
the $W_k$-oscillating walks of length~$n$
between two given points.
Throughout this section, we translate the region $W_k$ to apply the previous results. Namely, we use
\[
W_k= \{(x_1, \dots, x_k)\in\mathbb{Z}^k: x_1 > \cdots > x_k >0 \}.
\]
\begin{theorem}[Grabiner-Magyar~\cite{GrMa93}] For fixed~${\lambda},
  {\mu} \in W_k$, the exponential
  generating function $\widetilde O_{{\lambda}, {\mu}}(t)$ of the $W_k$-oscillating walks
from $\lambda$ to $\mu$, counted by their lengths,  satisfies
\[
\widetilde  O_{{\lambda}, {\mu}}(t) 
= \det \left( b_{{\mu_i}-{\lambda_j}}(2t)-b_{{\mu_i}+{\lambda_j}}(2t)\right)_{1\leq i, j\leq k}.
\]
\end{theorem}
We specialize the start and end positions as
${\lambda}=\delta=(k, k-1, \ldots, 1)$ and ${\mu}=
\delta+me_1=$ \mbox{$(k+m, k-1, \ldots, 1)$}. We are
interested in the sum over all values of $m$, and 
define  $\widetilde O_k(t)\equiv\sum_{m\geq 0}\widetilde  O_{{\delta}, me_1+\delta}(t)$.
We deduce the following. 
\begin{proposition}
  The exponential generating function for the class of oscillating
  tableaux ending with a row shape is the finite sum
\[
\widetilde{O}_k(t)= \sum_{u=0}^{k-1}(-1)^u \sum_{\ell=u}^{2k-1-2u} b_{\ell} \det (b_{i-j}-b_{kd-i-j})_{0 \leq i \leq k-1, i\neq u, 1\leq j \leq k-1}.
\]
\end{proposition}
This follows from the fact that the infinite sum which arises from
direct application of Grabiner and Magyar's formula telescopes after
applying the identity $b_{-k}=b_k$. The proof is technical, and largely an
exercise in tracking indices after applying co-factor expansion of the
determinants.

The bijection between the classes implies
$\widetilde{O}_k(t)=\widetilde{Y}_{2k}(t)$. Here are the first two
values, which are well known:
\[
\widetilde{O}_1(t)= \widetilde{Y}_2(t)= b_0+b_1, \quad \widetilde{O}_2(t) = \widetilde{Y}_4(t)=b_0^2+b_0b_1+b_0b_3-2b_1b_2-b_2^2-b_1^2+b_1b_3.
\]

\subsubsection{A Diagonal Expression}
\label{sec:diag}
\mbox{} Standard Young tableaux can also
be viewed as oscillating tableaux with no deleting steps. (The entries
tell you which box was added at a given time.) This gives us an
interesting correspondence between two lattice path classes. 
\begin{theorem}
  The set of oscillating lattice walks of length~$n$ in~$W_k$ starting
  at $\delta=(k, k-1, \dots, 1)$ and ending at the boundary at the
  boundary $\{me_1+\delta: m\geq 0\}$ is in bijection with the set of
  oscillating lattice walks of length $n$ in $W_{2k}$, using only
  positive steps ($e_j$), starting at $\delta$ and ending anywhere in
  the region.
\end{theorem}

We next obtain a new diagonal expression for standard Young tableaux
of bounded height. The expression is also a corollary of the
bijection.  We find the expression via the oscillating walks, and an
application of Gessel and Zeilberger's Weyl chamber reflectable walk
model.  The advantage of these diagonal representations is potential
access to asymptotic enumeration formulas, and possibly alternative
combinatorial representations. All of the generating functions are
D-finite, and we can use the work of~\cite{BoLaSa13} to determine
bounds on the shape of the annihilating differential equation.
\begin{theorem}
\label{thm:osc}
The ordinary generating function for oscillating walks starting at $\delta$ and
ending on the boundary $\{m\,e_1+\delta: m\geq 0\}$, is given by the
 formula
\[ O_k(t)=\Delta \left[ \frac{t^{2k-1}(z_3 z_4^2 \cdots
    z_k^{k-2})(z_1+1)\prod_{1\leq j<i \leq k} (z_i-z_j)(z_iz_j-1)
    \cdot \prod_{2 \leq i \leq k} (z_i^2 -1)}{1-t(z_1\cdots z_k)(z_1 +
    \oz_1 + \cdots z_k + \oz_k)}  \right]. \]
\end{theorem}
The proof of Theorem~\ref{thm:osc} is a rather direct application of
Gessel and Zeilberger's formula for reflectable walks in Weyl
chambers (the reader is directed to~\cite{GeZe92,GrMa93} for details).

\begin{proof}
As oscillating tableaux are counted by walks in the region $W_k := \{(x_1, \dots, x_d): x_1 > \dots > x_k > 0 \},$
taking steps
$ \{\pm e_1, \dots, \pm e_k \}, $
and starting at 
$ \delta = (k, k-1, \dots, 1), $
they are a reflectable walk model in the Weyl chamber $B_k$.   This Weyl chamber is generated by the simple roots
\[ \{ e_1-e_2, \dots, e_{k-1}-e_k, e_k\}, \]
and has a Weyl group $G$ of order $2^k \, k!$: the group has the
full action of $\mathfrak{S}_k$ on $\mathbb{R}^k$, along with negating
any subset of coordinates.  Any group element $\sigma$ which acts on
$\delta$ by permuting its entries (keeping all its entries positive)
has its order in the Weyl group the same as its order in
$\mathfrak{S}_k$.  Furthermore, negating $k$ entries of an element
corresponds to an even element of the Weyl group if and only if $k$ is
even.  Thus, if
\[ E(\bz) := (z_1\cdots z_k) \prod_{i<j}(z_i-z_j), \]
then 
\[ \sum_{w \in G} (-1)^{l(w)} \bz^{w(\delta)} = \sum_{I \subset \{1,\dots, k\}} (-1)^{|I|} E(\sigma_I(\bz)), \]
where $\sigma_I$ sends the element $z_j$ to $1/z_j$ if $j \in I$ and fixes it otherwise.  This sum simplifies to
\[ A(\bz) = \sum_{w \in G} (-1)^{l(w)} \bz^{w(\delta)} =
\frac{1}{(z_1\cdots z_k)^k} \prod_{1\leq j<i \leq k}
(z_i-z_j)(z_iz_j-1) \, \prod_{1 \leq i \leq k} (z_i^2 -1), \]
which can be proven by noting that both expressions represent
$(z_1\cdots z_k)^k A(\bz)$ as a polynomial of total degree $k(3k+1)/2$
with the same solutions (and then comparing the leading coefficients
of both expressions).

We are interested in walks that end on the boundary
$M= \{me_1+\delta:m \in \mathbb{N}\}$. The generating function of the
endpoints is 
\[ B(\bz) = \sum_{\mathbf{b} \in M} \bz^{-\mathbf{b}} = \frac{1}{z_1^{k-1} z_2^{k-1} z_3^{k-2} \cdots z_k (z_1-1)} .\]
Finally, the generating function of unrestricted walks starting at the
origin is
\[ C(\bz) = \frac{1}{1-t(z_1\cdots z_k)(z_1 + \oz_1 + \cdots z_k +
  \oz_k)}. \]
So by the classical result of Gessel and Zeilberger, we have the generating function
representation
\[ F(t) := \Delta \left( A(\bz) B(\bz) C(\bz) \right) ,
\]
which is exactly the formula given in the statement of this theorem.  
\end{proof}

\section{Tableau sequences as Baxter classes}
\label{sec:Baxter}
The combinatorial class that came to be known as Baxter permutations
was introduced in 1967 in a paper of Baxter~\cite{Baxter64} studying
compositions of commuting functions. A Baxter permutation of size~$n$
is a permutation~$\sigma\in \mathfrak{S}_n$ such that there are no
indices $i<j<k$ satisfying
$\sigma(j+1)<\sigma(i)<\sigma(k)<\sigma(j)\quad\text{or}\quad\sigma(j)<\sigma(k)<\sigma(i)<\sigma(j+1).
$
We shall denote by $B_n$ the number of Baxter permutations of size $n$
is $B_n$. They constitute entry A001181 of the Online Encyclopedia of
Integer Sequences (OEIS)~\cite{oeis}.

Chung, Graham, Hoggart and Kleiman~\cite{ChGrHoKl78} found the explicit formula 
\begin{equation}
\label{eq:BaxterFormula}
B_n=\sum_{k=1}^n\frac{\binom{n+1}{k-1}\binom{n+1}{k}\binom{n+1}{k+1}}{\binom{n+1}{1}\binom{n+1}{2}}.
\end{equation}

Many combinatorial classes have subsequently been discovered to have
the same counting sequence -- for example triples of lattice
paths~\cite{DuGu98} and plane bipolar orientations~\cite{Baxt01}. A recent
comprehensive survey of Felsner, Fusy, Noy and Orden~\cite{FeFuNoOr11}
finds many structural commonalities among these seemingly diverse
families of objects. Remarkably, there are intuitive bijections
connecting these classes, see for instance~\cite{BoBoFu09}.

The generating function of hesitating tableaux (i) in
Proposition~\ref{prop:ep} was determined by Xin and
Zhang~\cite{XiZh09}. Baxter numbers appear in their Table~3, and they
mention that the equivalence between the two could be proved by
applying creative telescoping to a formula for $B_n$ resembling the
one given in Equation~\eqref{eq:BaxterFormula} above.

Our contribution to this area is an explicit proof of that equivalence,
and an exploration of the connection between these classes and
the other well known Baxter classes. Clearly, the classes of
Proposition~\ref{prop:ep} have combinatorial bijections between them,
but they {\bf do not} share many of the properties of the other known
Baxter classes. However, each of them does have a natural subclass of
objects enumerated by Catalan numbers, as many Baxter families also
do. (For example, non-crossing partitions are counted by Catalan numbers.)

\begin{prop} 
  \label{thm:hes-baxter}
The following classes are in bijection:
\begin{enumerate}
\item[(i)] the set of hesitating tableaux with height bounded by $2$,
  starting with empty diagram, ending in a partition with a single
  part;
\item[(ii)] the set of open partition diagrams of length $n$ with
  no~enhanced $3$-crossing;
\item[(iii)]$W_2$-hesitating walks of length $n$ ending on the $x$-axis;
\item[(iv)]Baxter permutations of size $n+1$. 
\end{enumerate}
\end{prop}

Remark that Theorem~\ref{thm:Baxter} is simply the implication that
(i) and (iv) from Proposition~\ref{thm:hes-baxter} are in
bijection. We prove this with a generating function argument, and
deduce the other bijections using Proposition~\ref{prop:ep}.

\subsection{Proof of Theorem~\ref{thm:Baxter}}
We prove Theorem~\ref{thm:Baxter} and conjecture a stronger
result which could be useful to prove the bijection
combinatorially. This conjecture is partially verified using some of
the intermediary computations, so it is useful to have them made
explicit. We note that this is slightly different from both the proof that
appears in a previous version of this work~\cite{BuMeMi15} and from
the suggested proof of Xin and Zhang~\cite{XiZh09}.

We first set up some notation.  Let
$\overline{x}=\frac{1}{x}$, and consider the ring of formal series
$\mathbb{Q}[x, \overline{x}] [\![t]\!]$.  The operator~$CT_x$ extracts the constant term in $x$
of series of $\mathbb{Q}[x, \overline{x}] [\![t]\!]$. 

We recall the work of Bousquet-M\'elou and Xin~\cite{BoXi05}.  Here,
we only require the $k=2$ case from their work, and have
consequently eliminated some of the subscripts from the statements of their
results. Also, note that their definition of $W_2$ is shifted one unit
to the right, hence in the statement of their results, walks start at
$(1,0)$ rather than $(0,0)$. 

Let $Q$ denote the first quadrant in the plane,
$Q=\{(x, y): x, y \geq 0\}$, and let $W_2$ denote the region
$W_2=\{(x, y): x>y \geq 0\}$. Walks taking $n$ steps that start at
$\lambda$ and end at $\mu$ and remain in $Q$ and $W_2$ are, respectively,
denoted by~$q(\lambda, \mu, n)$ and $w(\lambda, \mu, n)$.

Bousquet-M\'elou and Xin's Proposition 12 in~\cite{BoXi05}, based on a
classic reflection argument, implies the following.  For any starting
and ending points $\lambda$ and $\mu$ in $W_2$, the number of
$W_2$-hesitating walks going from $\lambda$ to $\mu$ can be expressed
in terms of the number of $Q$-hesitating walks:
\[
w(\lambda, \mu, n) = q(\lambda, \mu, n)-q(\lambda, \overline{\mu}, n)
\]
where $\overline{(x,y)}=(y,x)$.  They define a simple sign reversing
involution between pairs of walks; the walks restricted to $W_2$ appear
as fixed points.

We consider the following two generating functions for $Q$-hesitating walks that
start at $(1,0)$ and end on an axis:
\[
H(x;t) = \sum_{ i\geq 1, n\geq 0 } q((1,0), (i,0), 2n) x^i\,t^n\quad \text{and}\quad 
V(y;t) = \sum_{ i\geq 1, n\geq 0 } q((1,0), (0,i), 2n) y^i\,t^n.
\]
By applying the proposition we see immediately that the bivariate generating
function $W(x;t)$ for $W_2$-hesitating walks that start at $(0,1)$ and end on the
$x$- axis satisfies the formula
\begin{equation}
\label{eqn:hv}
W(x;t)=\sum_{i\geq 1, n\geq 0} w_2((1,0), (i,0), 2n)\, x^it^n = H(x;t)-V(x;t).
\end{equation}

Theorem~\ref{thm:Baxter} is equivalent to the statement
\begin{equation}
\label{eq:Baxter}
W(1;t)=\sum B_{n+1} t^n.
\end{equation}

\begin{proof}[Proof of Theorem~\ref{thm:Baxter}]
 For $i\geq 0$, Bousquet-M\'elou and Xin~\cite{BoXi05} show 
 the following
\begin{eqnarray}
[x^{i+1}]H(x;t)&=&\mathrm{CT}_x\frac{Y}{t(1+x)}\bar{x}^{2+i}(x^2-\bar{x}^2Y^2+\bar{x}^3Y)\\ \mbox{}
[x^{i+1}]V(x;t)&=&\mathrm{CT}_x\frac{Y}{t(1+x)}x^{3+i}(x^2-\bar{x}^2Y^2+\bar{x}^3Y).
\end{eqnarray}
Thus, we deduce
\[
[x^{i+1}]W(x;t)=[x^{i+1}]H(x;t)-V(x;t) =\mathrm{CT}_x\frac{Y}{t(1+x)}(\bar{x}^{2+i}-x^{3+i})(x^2-\bar{x}^2Y^2+\bar{x}^3Y).\]
Hence we have $$W(1;t)=T(t)+U(t)+V(t),$$ 
where 
\begin{eqnarray*}
T(t)&=&\mathrm{CT}_x\sum_{i\geq 0}\frac{Y}{t(1+x)}(\bar{x}^i-x^{5+i}),\\
U(t)&=&-\mathrm{CT}_x\sum_{i\geq 0}\frac{Y^3}{t(1+x)}(\bar{x}^{4+i}-x^{1+i}),\\
V(t)&=&\mathrm{CT}_x\sum_{i\geq 0}\frac{Y^2}{t(1+x)}(\bar{x}^{5+i}-x^{i}).
\end{eqnarray*}
Then we use the following identity from~\cite{BoXi05} (valid for $k\geq 1$
and $\ell\in\mathbb{Z}$)
$$\mathrm{CT}_x\frac{Y^k}{t(1+x)}\bar{x}^{\ell}=\mathrm{CT}_x\frac{Y^k}{t(1+x)}x^{\ell-k+1},$$
which gives the following simplifications:
\begin{eqnarray*}
T(t)&=&\mathrm{CT}_x\sum_{i\geq
        0}\frac{Y}{t(1+x)}(x^i-x^{5+i})=\mathrm{CT}_x\frac{Y}{t(1+x)}(1+x+x^2+x^3+x^4),\\
U(t)&=&-\mathrm{CT}_x\sum_{i\geq
        0}\frac{Y^3}{t(1+x)}(x^{2+i}-x^{1+i})=\mathrm{CT}_x\frac{Y^3}{t(1+x)}x,\\
V(t)&=&\mathrm{CT}_x\sum_{i\geq 0}\frac{Y^2}{t(1+x)}(x^{4+i}-x^{i})=\mathrm{CT}_x\frac{Y^2}{t(1+x)}(-1-x-x^2-x^3).
\end{eqnarray*}
Hence, defining
$A_{\ell,k}(t)=\mathrm{CT}_x\frac{Y^k}{t(1+x)}x^{\ell}$, we can
collect terms to obtain
\begin{equation}
W(1;t)=\sum_{r=0}^4A_{r,1}(t)+A_{1,3}(t)-\sum_{r=0}^3A_{r,2}(t).
\label{eq:WasA}
\end{equation}
It is shown in~\cite{BoXi05} that the Lagrange inversion formula yields, for $n\in\mathbb{N}$,
$$
[t^n]A_{\ell,k}(t)=\sum_{j\in\mathbb{Z}}a_n(\ell,k,j),\quad
\text{with}\quad a_n(\ell,k,j)=\frac{k}{n+1}\binom{n+1}{j}\binom{n+1}{j+k}\binom{n}{j-\ell}.
$$
Here we apply the convention $\binom{n}{j}=0$ for $j<0$ or $j>n$.

Next, it is straightforward to detect and check the linear relations (valid for
$n\in\mathbb{N}$ and $j\in\mathbb{Z}$)
$$
a_n(4,1,n-j+2)+a_n(1,3,j-1)-a_n(2,2,n-j+1)-a_n(3,2,j)=0,
$$
$$
a_n(1,1,n-j)+a_n(2,1,j+1)-a_n(0,2,j)=0,
$$
which respectively give $A_{4,1}(t)+A_{1,3}(t)-A_{2,2}(t)-A_{3,2}(t)=0$
and $A_{1,1}(t)+A_{2,1}(t)-A_{0,2}(t)=0$. Remarkably,
expression~\eqref{eq:WasA} for $W(1;t)$ simplifies to
$$
W(t)=A_{0,1}(t)+A_{3,1}(t)-A_{1,2}(t).
$$
For $n\geq 1$, the Baxter number $B_n$ is given by
$B_n=\sum_{j\in\mathbb{Z}}b_{n,j}$, with
$b_{n,j}=\frac{\binom{n+1}{j-1}\binom{n+1}{j}\binom{n+1}{j+1}}{\binom{n+1}{1}\binom{n+1}{2}}$,
and again it is easy to detect and check that (for $n\in\mathbb{N}$
and $j\in\mathbb{Z}$)
$$
a_n(0,1,j)+a_n(3,1,j+1)-a_n(1,2,j)=b_{n+1,j+1},
$$
so that $A_{0,1}(t)+A_{3,1}(t)-A_{1,2}(t)=\sum_{n\geq 0}B_{n+1}t^n$,
and thus $[t^n]W(1;t)=B_{n+1}$.
\end{proof}



\subsection{Consequence: a new generating tree}
\label{sec:GeneratingTree}
A \emph{generating tree\/} for a combinatorial class expresses recursive
structure in a rooted plane tree with labeled nodes. The objects of
size~$n$ are each uniquely generated, and the set of objects of size $n$
comprise the $n$th level of the tree. They are useful for
enumeration, and for showing that two classes are in bijection. Theorem~\ref{thm:Baxter} yields a new generating tree
construction for Baxter objects.

Several different formalisms exist for generating trees,
notably~\cite{BaBoDeFlGaGo02}. The central properties are as
follows. Every object $\gamma$ in a combinatorial class $\mathcal{C}$ is
assigned a label $\ell(\gamma)\in\mathbb{Z}^k$, for some fixed
$k$. There is a rewriting rule on these labels with the property that
if two nodes have the same label then the ordered list of labels of
their children is also the same.  We consider labels that are pairs of
positive integers, specified by~$\{\ell_\text{Root}: [i,j]
\rightarrow \operatorname{Succ}([i,j])\}$, where $\ell_\text{Root}$
is the label of the root.

Two generating trees for Baxter objects are known in the literature,
and one consequence of Theorem~\ref{thm:Baxter} is a third, using the
generating tree for open partitions given by Burrill~\emph{et
  al.}~\cite{Buetal12}. This tree differs from the other two already
at the third level, illustrating a very different decomposition of the
objects.  For the three different systems we give the succession
rules, and the first 5 levels of the tree (unlabelled), in
Figure~\ref{fig:trees}.
\begin{figure}\center
\begin{subfigure}[b]{0.45\textwidth}
                \includegraphics[width=\textwidth, height=2cm]{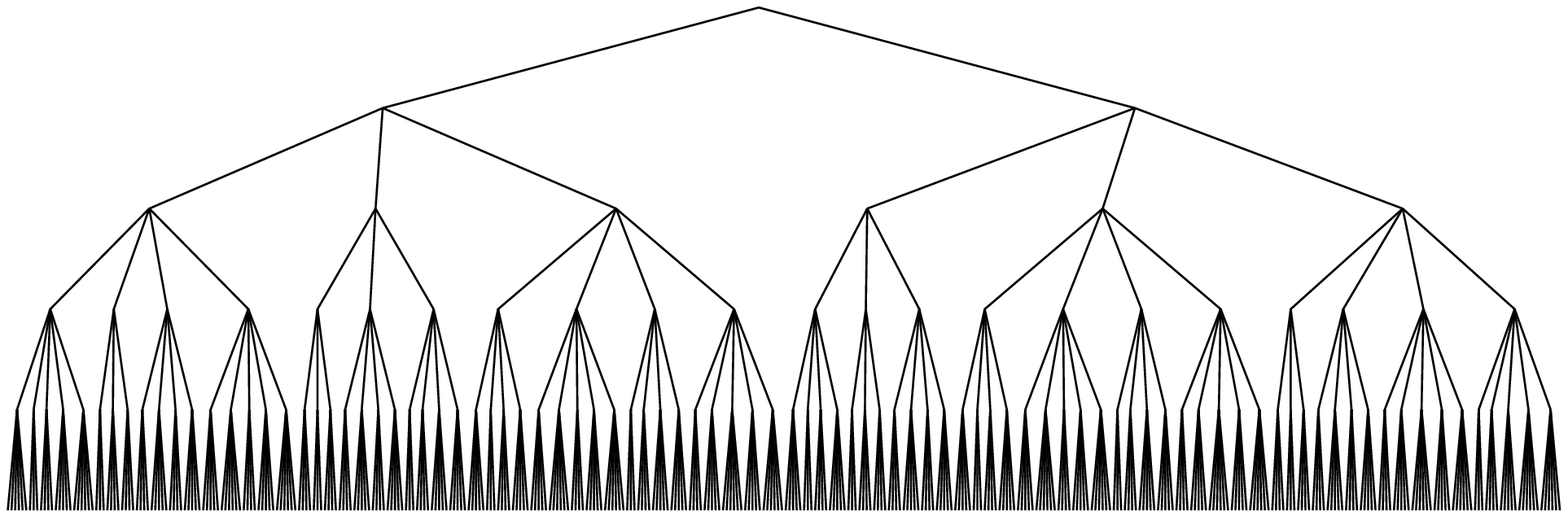}

{\tiny $\{[1,1];[i,j]\rightarrow [1,j+1], \dots, [i, j+1], [i+1, j], \dots [i+1, 1]\}$}
\end{subfigure}\begin{subfigure}[b]{0.45\textwidth}
                \includegraphics[width=\textwidth, height=2cm]{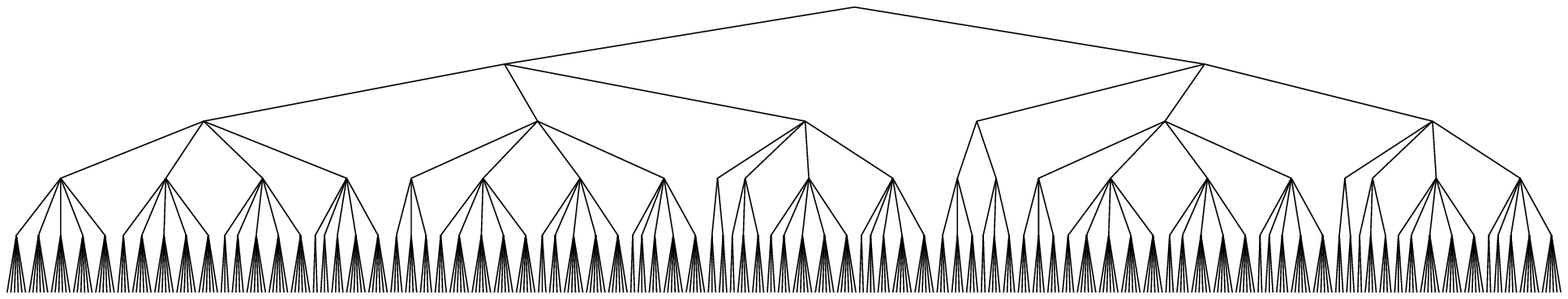}

{\tiny$\{[0,2];
[i,j]\rightarrow [0, j], \dots , [i-1, j], [1, j+1], \dots, [i+j-1, 2]\}$}
\end{subfigure}

\begin{subfigure}[b]{0.45\textwidth}
                \includegraphics[width=\textwidth,height=2cm]{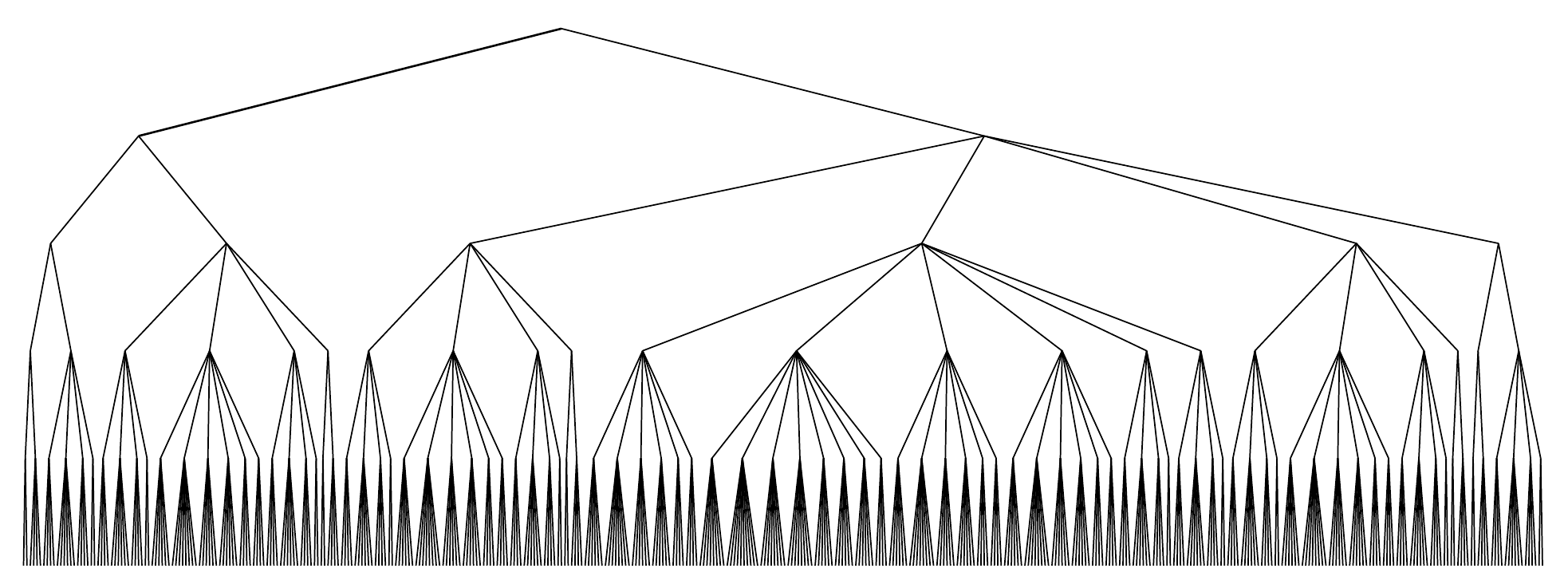}
{\tiny \[
\begin{array}{rll}
\{[0,0];[i,j]\rightarrow
&[i,i], [i +1,j]\\
&[i,j],[i,j+1],\dots,[i,i-1], &\mathrm{if}\, i>0\\
&[i-1,j],[i-1,j+1],\dots,[i-1,i-1], &\mathrm{if}\,i>0 \\ 
&[i , j -1], [i-1, j-1]&\mathrm{if}\, i > 0, \mathrm{and}\,  j > 0\}.\\
\end{array} \]}

\end{subfigure}

 \caption{The first five levels of each of the Baxter generating
   trees. They are respectively from~\cite{BoBoFu09}~\cite{BoGu14}~\cite{Buetal12}. }
\label{fig:trees}

\mbox{}\\
{\color{gray!50} \hrule}
\end{figure}

\subsection{A conjectured refinement}
\label{sec:refinement}

We have proved that the coefficients $a(n,m)$
counting $W_2$-hesitating walks of length $2n$ from $(0,0)$ to $(m,0)$
satisfy $\sum_ma(n,m)=B_{n+1}$, with $B_n$ the $n$th Baxter number. A
bijective proof is yet to be found, and in that perspective a natural
question is whether the parameter $m$ corresponds to a simple
parameter on another Baxter family.

\begin{proposition}
  The family of $Q$-hesitating excursions, that is to say the hesitating walks in the lattice $Q=\{(x, y): x, y \geq 0\}$ starting and ending at the origin, is a Baxter family: the number of such walks of length $2n$ is equal to $B_{n+1}$.
\end{proposition}
\begin{proof}
We show an easy bijection with the set $\cT_n$ of non-intersecting triples of lattice paths
each of length $n$ with steps either $N=(0,1)$ (north steps) or $E=(1,0)$ (east steps), with respective starting points $(-1,1)$, $(0,0)$,
$(1,-1)$ and respective ending points $(k-1,n-k+1)$, $(k,n-k)$, $(k+1,n-k-1)$ for some 
$k \in \{0,\dots,n\}$. 
For $3$ distinct points $p_1,p_2,p_3$ in $\mathbb{Z}^2$ on a same line of slope $-1$, ordered from top-left to bottom-right, 
define the \emph{distance-pair\/} for $(p_1,p_2,p_3)$
as the pair $(i,j)$ of nonnegative integers such that $x(p_1)=x(p_2)-i-1$ and $x(p_3)=x(p_2)+j+1$. Let $(P_1,P_2,P_3)\in\cT_n$. 
For $r \in \{0,\dots,n\}$ and $i\in\{1,2,3\}$,
let $p_i^{(r)}$ be the point on $P_i$ after $r$ steps, and let $d(r)$ be the distance-pair
for $(p_1^{(r)},p_2^{(r)},p_3^{(r)})$; note that $d(0)=(0,0)$ and $d(n)=(0,0)$ and that $d(r)\in q$ for $0\leq r\leq n$. 
Moreover, for $0<r\leq n$, the vector $\delta(r):=d(r)-d(r-1)$ is in the set $\{(\pm 1,0),(0,\pm 1),(1,-1),(-1,1),(0,0)\}$,
with two possibilities for being $(0,0)$ (whether the $r$th steps in $P_1,P_2,P_3$ are all north or all east). Hence the situation
for the successive distance-pairs $d_0,\ldots,d_n$ is exactly the same as for the successive points of even rank in a $Q$-excursion of length $2n$. 
Figure~\ref{fig:baxterwalks} (left and middle) illustrates this bijection.
\end{proof}

\begin{figure}[h!]
\begin{center}
\includegraphics[width=0.92\textwidth]{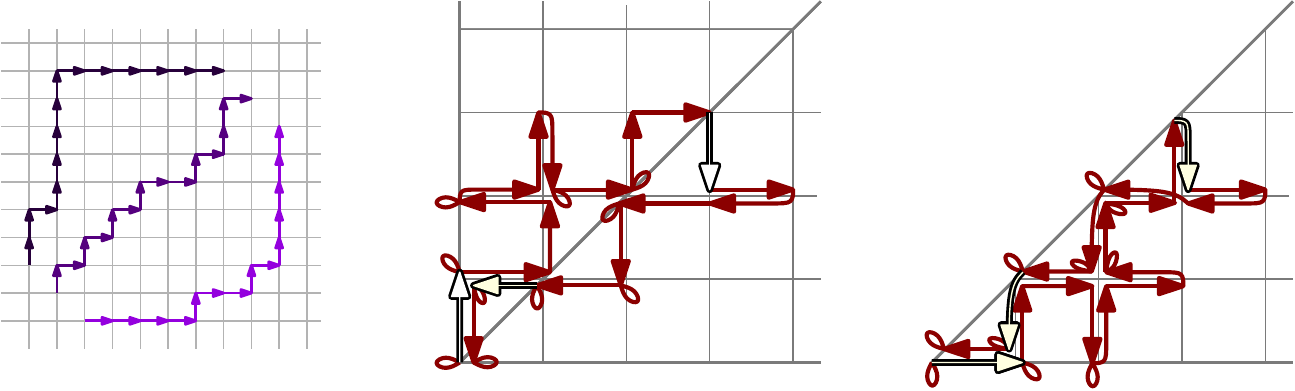}
\end{center}
\caption{\textit{Left.} A non-intersecting triple of lattice paths. \textit{Middle.} A $Q$-hesitating excursion. The stay steps are drawn as loops. The switch-multiplicity of the walk is $3$ (the white arrows indicate the marked steps). \textit{Right.} A $W_2$-hesitating excursion with $3$ marked steps each leaving the diagonal. These three objects are in correspondence.  }
\label{fig:baxterwalks}
\mbox{}\\
{\color{gray!50} \hrule}
\end{figure}

We now define a secondary parameter $m$ for $Q$-excursions. Let $w$
be a $Q$-excursion of length $2n$ where $e_r$ denotes the $r$th
step, for $1\leq r\leq 2n$. Consider, if any, the first step $e_{i_1}$
that visits the region $x<y$. Then consider, if any, the first step
$e_{i_2}$ after $e_{i_1}$ that visits the region $x>y$, and so on
(switching between $x<y$ and $x>y$ each time). We have here a stopping
iterative process yielding, for some $m\geq 0$, $m$ marked steps
$e_{i_1},\ldots,e_{i_m}$ with $i_1<\ldots<i_m$; $m$ is called the
\emph{switch-multiplicity} of the excursion.
For instance, the switch-multiplicity of the excursion at the middle
of Figure~\ref{fig:baxterwalks} is $3$.  Note also that $m\leq n$
since two marked steps cannot be consecutive (the case $m=n$ is
reached by the unique excursion where $i_1=1,i_2=3,i_3=5,\ldots$,
i.e., the excursion that alternates pairs of steps $(0,1),(0,-1)$ with
pairs of steps $(1,0),(-1,0)$).  Denote by $q(n,m)$ the
number of $Q$-hesitating excursions of length $2n$ and
switch-multiplicity $m$, and $a(n,m)$ the
number of $W_2$-hesitating walks of length $2n$ from $(0,0)$ to $(m,0)$.

\begin{conj}
\label{thm:BaxterConjecture}
For $n,m\geq 0$, we have $
q(n,m)=a(n,m).$
\end{conj}

We have thought of the switch-multiplicity as a natural candidate
because of the analogy with a well-known bijection between excursions
of length $2n$ on the line $\mathbb{Z}$ and walks of length $2n$
starting at $0$ on the half-line $\mathbb{Z}_{\geq 0}$ (with steps in
$\pm 1$ for both types of walks), where a similar switch-multiplicity
parameter for excursions (this time switching between
$\mathbb{Z}_{<0}$ and $\mathbb{Z}_{>0}$) corresponds to half the
ending abscissa of positive walks. However, what surprises us is that,
while we have strong evidence the conjecture is true, we do not even
have a proof for $m=1$ (the case $m=0$ is trivial).

Let us now slightly reformulate the conjecture so that we have $W_2$-hesitating walks on both sides. Consider a $Q$-hesitating
excursion $w$ of length $2n$, with $e_{i_1},\ldots,e_{i_m}$ the steps
given by the stopping iterative process
(switching between $x<y$ and $x>y$). Accordingly $w$ splits into a concatenated sequence of $m+1$ parts $\pi_0,\ldots,\pi_{m}$,
where $\pi_0$ is the part before $e_{i_1}$ ($\pi_0=w$ if $m=0$), for $1\leq h<m$, 
$\pi_h$ 
is the part between $e_{i_h}$ (included) and $e_{i_{h+1}}$ (excluded), 
and for $m\geq 1$, $\pi_{m}$
 is the ending part of $W_2$ starting from $e_{i_m}$. Each walk $\pi_i$ starts and ends
on the diagonal $x=y$ and stays in $x\geq y$ for $i$ even and in $x\leq y$ for $i$ odd. Hence, if we reflect each odd walk $\pi_{2i+1}$
according to the diagonal $x=y$,  
we obtain a $W_2$-hesitating walk from $(0,0)$ to $(0,0)$ with $m$ marked steps (the steps at positions $i_1,\ldots,i_m$)
each leaving the diagonal $x=y$. In addition, due to recording the marked steps, there is no loss of information (the original
excursion can be recovered). 
This correspondence is depicted by Figure~\ref{fig:baxterwalks}
(middle and right).  Hence, if we denote by $a(n;i,j,m)$ the number of
$W_2$-hesitating walks of length $2n$ from $(0,0)$ to $(i,j)$ with
$2n$ steps and $m$ marked steps each leaving the diagonal $x=y$,
Conjecture~\ref{thm:BaxterConjecture} can be reformulated as:

\begin{conj}[Reformulation]
For $n,m\geq 0$, we have
$$
a(n;m,0,0)=a(n;0,0,m).
$$
\end{conj}
Actually, an even stronger symmetry seems to hold:

\begin{conjecture}
\label{conj:super}For $n,i,j\geq 0$, we have
$$
a(n;i,0,j)=a(n;j,0,i).
$$   
\end{conjecture}

Note that there is clearly a one-to-one correspondence between steps
leaving the diagonal $x=y$ and steps ending at the diagonal $x=y$.
Hence $a(n;i,j,m)$ is also the number of $W_2$-hesitating walks of
length $2n$ from $(0,0)$ to $(i,j)$ with $2n$ steps and $m$ marked
steps each ending at the diagonal $x=y$. In that form it is easy to
obtain a recurrence for the coefficients $a(n;i,j,m)$ by considering
the effect of adding the last two steps (note that each of the two
last steps has to be unmarked if empty or not ending at $x=y$, and
might be either unmarked or marked if non-empty and ending at
$x=1$). Denote by $\cS$ the set of steps $\{(\pm 1,0), (0,\pm 1),
(-1,1), (1,-1)\}$ together with the two stationary steps $s_1$ and
$s_2$, where $s_1$ simulates taking the pair $(1,0)$ and $(-1,0)$ as a
single step, and $s_2$ simulates taking the pair $(0,1)$ and $(0,-1)$
as a single step.  We have the following recurrence (with $\delta$ the
Kronecker symbol), from which we have been able to check that
Conjecture~\ref{conj:super} holds for all $n, i, j$ with $n\leq
56$:
\begin{itemize}
\item for $n=0$,
\begin{eqnarray*}
a(n;i,j,m)&=&1\ \ \mathrm{if}\ i=j=m=0,\\
a(n;i,j,m)&=&0\ \ \mathrm{otherwise},
\end{eqnarray*}
\item
for $n>0$, 
\begin{eqnarray*}
a(n;i,j,m)&=&0\ \ \ \ \mathrm{for}\ \ (i,j,m)\notin\cD:=\{0\leq j\leq i,\ 0\leq m\leq n\},\\
a(n;i,j,m)&=&\delta_{i=j}\cdot\sum_{s\in\cS\backslash s_1}a(n-1;i-x(s),j-y(s),m)\\
&&+\delta_{i>j}\cdot\sum_{s\in\cS}a(n-1;i-x(s),j-y(s),m)\\
&&+\delta_{i=j}\cdot\sum_{s\in\cS\backslash s_1}a(n-1;i-x(s),j-y(s),m-1)\\
&&+\delta_{i=j+1}\cdot a(n-1;i,j,m-1)\ \ \ \ \mathrm{for}\ (i,j,m)\in\cD.\\
\end{eqnarray*}
\end{itemize}

\section{Conclusion}
\label{sec:Conclusion}
We conclude with a few thoughts on future directions. Starting with
the correspondences described here, we can easily give an expression
for the generating function of standard Young tableaux of bounded
height as a diagonal of a rational function. The original proofs of
the D-finiteness of these generating functions (for arbitrary $k$) of
Gessel~\cite{Gess90} used a diagonal type operation over series with
infinite variable sets. We are interested in the conjecture of
Christol which states that globally bounded D-finite series can always
be expressed as a diagonal of a rational function. Perhaps the key to
understanding how to find such diagonal expressions lurks in Weyl
chamber walk representations of combinatorial objects, since they
easily yield diagonal expressions from the machinery of Gessel and
Zeilberger.

Baxter numbers generalize, in some sense, Catalan numbers.  Both are
ubiquitous combinatorial sequences, and both are related to hesitating
walk families. Are hesitating walks in higher dimensions similarly
common?
\section*{Acknowledgements}
We are extremely grateful to Sylvie Corteel, Lily Yen, Yvan le Borgne,
Sergi Elizalde, and Guillaume Chapuy for stimulating conversations,
and important insights. JC is supported by the ANR \textit{GRAAl},
ANR-14-CE25-0014-02, and by the PIMS postdoctoral fellowship
grant. The work of \'EF was partly supported by the ANR grant
\textit{Cartaplus} 12-JS02-001-01 and the ANR grant \textit{EGOS}
12-JS02-002-01.  SM is supported by an NSERC Alexander Graham Bell
Canada Graduate Scholarship. The work of MM is partially supported by
an NSERC Discovery Grant.

\small
\bibliographystyle{plain}
\bibliography{main} 

\end{document}